\documentclass[reqno, 11pt]{amsart}
\usepackage{amssymb}
\usepackage[usenames, dvipsnames]{color}
\usepackage{todonotes}
\usepackage{hyperref}
\usepackage{verbatim}
\usepackage{accents} 

\numberwithin{equation}{section}

\newtheorem{theorem}{Theorem}[section]

\newtheorem{lemma}[theorem]{Lemma}

\theoremstyle{definition}
\newtheorem{remark}[theorem]{Remark}

\theoremstyle{definition}
\newtheorem{definition}[theorem]{Definition}

\theoremstyle{definition}
\newtheorem{assumption}[theorem]{Assumption}

\makeatletter
\def\dashint{\operatorname%
{\,\,\text{\bf--}\kern-.98em\DOTSI\intop\ilimits@\!\!}}
\makeatother

\newcommand{\vertiii}[1]{{\left\vert\kern-0.25ex\left\vert\kern-0.25ex\left\vert #1 
    \right\vert\kern-0.25ex\right\vert\kern-0.25ex\right\vert}}

\def\.5{\frac{1}{2}}

\def\bR{\mathbb{R}}

\def\cB{\mathcal{B}}

\def\cD{\mathcal{D}}

\def\cF{\mathcal{F}}

\def\cP{\mathcal{P}}
\def\cM{\mathcal{M}}

\def\cQ{\mathcal{Q}}

\def\cS{\mathcal{S}}

\def\cU{\mathcal{U}}
\def\cV{\mathcal{V}}
\def\cW{\mathcal{W}}

\def\cU{\mathcal{U}}

\begin{document}
\title[Parabolic equations with a half-time derivative]{$L_p$-estimates for parabolic equations in divergence form with a half-time derivative}

\author[P. Jung]{Pilgyu Jung}
\address[P. Jung]{Department of Mathematics, Sungkyunkwan University, 2066 Seobu-ro, Jangan-gu, Suwon-si, Gyeonggi-do, 16419, Republic of Korea}

\email{\href{mailto:pilgyujung@skku.edu}{\nolinkurl{pilgyujung@skku.edu}}}


\author[D. Kim]{Doyoon Kim}
\address[D. Kim]{Department of Mathematics, Korea University, 145 Anam-ro, Seongbuk-gu, Seoul, 02841, Republic of Korea}

\email{\href{mailto:doyoon_kim@korea.ac.kr}{\nolinkurl{doyoon_kim@korea.ac.kr}}}




\begin{abstract}
We establish the unique solvability of solutions in Sobolev spaces to linear parabolic equations in a more general form than those in the literature.
A distinguishing feature of our equations is the inclusion of a half-order time derivative term on their right-hand side.
We anticipate that such equations will prove useful in various problems involving time evolution terms.
Notably, the coefficients of the equations exhibit significant irregularity, being merely measurable with respect to the temporal variable or one spatial variable.
\end{abstract}

\maketitle

\section{Introduction}

This paper is devoted to studying  $L_p$ theory of second-order parabolic equations of the following form:
\begin{equation}
							\label{intro_maineq}
u_t-D_i(a_{ij}D_j u)+\lambda u= D_t^{1/2}h + D_ig_i+f,
\end{equation}
defined in $\bR \times \bR^d$, where $D_t^{1/2}h$ is the half-time derivative of $h$; for the definition, see \eqref{closed form} and Definition \ref{weak form}.
As one may notice, the right-hand side of \eqref{intro_maineq} includes the term $D_t^{1/2}h$, whereas the usual parabolic equations in divergence form considered in the literature (see, for instance, \cite{MR0241822, MR1465184, MR2304157, MR2764911}) are of the form:
    \begin{equation}
        \label{intro_classiceqn}
    u_t-D_i(a_{ij}D_j u)+\lambda u=D_ig_i+f.       
    \end{equation}
Thus, the emphasis of this paper is on whether there exists a unique solvability for parabolic equations with a half-order time derivative on the right-hand side when the solution spaces are Sobolev spaces.
In particular, we establish the following $L_p$-estimate:
\begin{equation}
							\label{eq0701_02}
\|D_t^{1/2}u\|_p + \|Du\|_p + \sqrt{\lambda} \|u\|_p \lesssim \|h\|_p + \|g\|_p + \frac{\|f\|_p}{\sqrt{\lambda}},
\end{equation}
where $\|\cdot\|_p = \|\cdot\|_{L_p(\bR \times \bR^d)}$ with $p \in (1,\infty)$ under appropriate assumptions on $a_{ij}$.
Recall that the usual $L_p$-estimate for \eqref{intro_classiceqn} is
    \begin{equation*}
        \|Du\|_p+\sqrt{\lambda }\|u\|_p \lesssim \|g\|_p+\frac{\|f\|_p}{\sqrt{\lambda}}.
    \end{equation*}
Upon observing the term $\|D_t^{1/2}u\|_p$ on the left-hand side of the estimate \eqref{eq0701_02}, even when $h=0$, it becomes evident that our estimate provides more sophisticated information about solutions to parabolic equations.

To the best of our knowledge, the term $D_t^{1/2}h$ was first considered in the work of Kaplan \cite{MR0200593}, where the author introduced an expression equivalent to $D_t^{1/2}h$ to address a boundary value problem for parabolic equations.
More precisely,
given that the $L_2$-solvability of equations as in \eqref{intro_classiceqn} with {\em homogeneous initial and lateral boundary conditions} is well-established, Kaplan considered parabolic equations with an inhomogeneous lateral boundary condition $v$ on $(0,T) \times \partial\Omega$, where $v$ is defined on $[0,T] \times \overline{\Omega}$.
If $\tilde{u}$ is a desired solution to the equation
    \begin{equation*}
    \tilde{u}_t-D_i(a_{ij}D_j \tilde{u})+\lambda \tilde{u}=D_i\tilde{g_i}+\tilde{f}
    \end{equation*}
in $(0,T) \times \Omega$ with the inhomogeneous lateral boundary condition $\tilde{u} = v$ on $(0,T) \times \partial\Omega$ (and with the zero initial condition), then $u := \tilde{u} - v$ satisfies
\begin{equation}
							\label{eq0701_03}
u_t-D_i(a_{ij}D_j u)+\lambda u=-v_t+D_ig_i+f
\end{equation}
in $(0,T) \times \Omega$ with $u = 0$ on $\left(\{0\}\times \Omega\right) \cup \left((0,T) \times \partial\Omega\right)$ (i.e., homogeneous boundary condition),
where
    \[
    g_i=\tilde{g_i}+a_{ij}D_j v,\quad f=\tilde{f}-\lambda v.
    \]
Then, one can apply the already established $L_2$-solvability for equations with {\em homogeneous boundary conditions} to the equation \eqref{eq0701_03} to find $u$ (and then $\tilde{u}$), provided that $v_t$ belongs to $L_2((0,T)\times \Omega)$ or to the space $L_2((0,T); V')$, where $V'$ is the dual of $V := H^1_0(\Omega)$.
However, unless $v$ is sufficiently smooth, $v_t$ does not belong to $L_2\left((0,T) \times \Omega\right)$.
Instead of investigating if $v_t$ belongs to $L_2((0,T); V')$, in \cite{MR0200593} Kaplan treated the term $v_t$, from our perspective, as $D_t^{1/2}h$, where $h \in L_2\left((0,T) \times \Omega\right)$.
This transformation turned the equation \eqref{eq0701_03} into \eqref{intro_maineq} with homogeneous boundary conditions.
Then, using the Lax-Milgram theorem, not the Galerkin method usually employed for parabolic equations, Kaplan proved the unique solvability of solutions to \eqref{intro_maineq} in the $L_2$-space setting.
In particular, the following $L_2$-estimate is proved.
    \begin{equation*}
         \|D_t^{1/2}u\|_2+\|Du\|_2+\sqrt{\lambda }\|u\|_2
         \lesssim\|h\|_2+ \|g\|_2+\frac{\|f\|_2}{\sqrt{\lambda}},
    \end{equation*}
where $\|\cdot\|_2 = \|\cdot\|_{L_2}$.

Half-time derivatives of solutions to parabolic equations are also considered in \cite{MR0241822}, where, although the authors did not consider equations as in \eqref{intro_maineq}, they proved that solutions to the usual parabolic equations as in \eqref{intro_classiceqn} belong to $W_2^{1/2,1}\left((0,T) \times \Omega \right)$, an element of which, when the domain is $\bR \times \Omega$, has a finite integral
\[
\vertiii{u}= \left(\int_0^\infty \ell^{-2} \|u(\cdot+\ell,\cdot) - u(\cdot,\cdot)\|_{L_2(\bR\times\Omega)}^2 \, d\ell\right)^{1/2}.
\]
In this case, one can verify that $\vertiii{u}$ is equivalent to $\|D_t^{1/2}u\|_{L_2(\bR \times \Omega)}$.

After a long period since the appearance of the works \cite{MR0200593, MR0241822}, there has been a renewed exploration of half-time derivatives of solutions to parabolic equations, along with Kaplan's argument of using the Lax-Milgram theorem for the $L_2$-solvability of parabolic equations instead of the Galerkin method.
Thus, half-time derivatives of solutions are now being examined in various contexts related to parabolic equations.
For instance, Kaplan's argument was used in \cite{MR4127944} for the $L_2$-solvability of boundary value problems for homogeneous parabolic systems, and in \cite{MR3538523, MR3707301} for the maximal regularity of non-autonomous parabolic equations.
See also \cite{MR3906170}, where the authors utilized half-time derivatives to establish regularity properties (including the self-improvement of the integrability of spatial gradients) of weak solutions to parabolic systems, and \cite{dindovs2023regularity}, where the author investigated half-time derivatives in proving regularity results for solutions to boundary value problems of homogeneous parabolic equations (having a drift term) with a boundary condition whose half-time derivative is in $L_p$ on the lateral boundary.

Regarding $L_p$ theory for \eqref{intro_maineq} together with estimates as in \eqref{eq0701_02}, as mentioned earlier, the results are established in \cite{MR0200593} for $p=2$.
However, for general $p\in (1,\infty)$ and $h\neq0$, with appropriate regularity assumptions on the leading coefficients $a_{ij}$, there do not seem to be results available.
In this paper, we provide results concerning \eqref{intro_maineq} that correspond to the $L_p$ theory of the usual parabolic equations \eqref{intro_classiceqn}.

Our motivation for dealing with equations as in \eqref{intro_maineq} within the framework of $L_p$ spaces is twofold: first, as hinted above, to generalize the known $L_2$ theory to $L_p$ theory for all $p \in (1,\infty)$, and, second, to deepen our understanding of solutions to the usual parabolic equations as in \eqref{intro_classiceqn}.
Concerning the latter, insights into half-time derivatives of solutions to parabolic equations can be applied to various contexts.
For instance, when establishing the unique $L_p$ solvability of the usual parabolic equations as in \eqref{intro_classiceqn} in $(0,T) \times \Omega$, where $\Omega$ has a very irregular boundary $\partial\Omega$, one may confront equations as in \eqref{intro_classiceqn} with $u_tI_{\Omega^*}$ on the right-hand side, where $\Omega^* \subset \Omega$.
By the well-known theory for parabolic equations, one can say that, if the lateral boundary condition is zero, $u_t \in L_p\left((0,T); H_p^{-1}(\Omega)\right)$, where $H_p^{-1}(\Omega)$ is the dual of $\mathring{W}_p^1(\Omega)$.
However, it is not clear if $u_tI_{\Omega^*}$ belongs to the same space.
As seen in \eqref{eq0523_01} and Lemma \ref{lem0514_1}, if $H$ denotes the Hilbert transform with respect to time (see Section \ref{sec01}), we may write $u_t = D_t^{1/2}h$, where $h = - H(D_t^{1/2}u)$, provided that $D_t^{1/2}u \in L_p$.
Then, we obtain equations as \eqref{intro_maineq}.
More precisely, we have
\[
u_t-D_i(a_{ij}D_j u)+\lambda u= D_t^{1/2}(h I_{\Omega^*}) + D_ig_i+f,
\]
and, if the volume of $\Omega^*$ is sufficiently small, 
the $L_p$-norm of the term $h I_{\Omega^*}$, which is equivalent to that of $D_t^{1/2}u I_{\Omega^*}$, can be eventually absorbed into the left-hand side of the estimates as in \eqref{eq0701_02} by using the smallness of the volume of $\Omega^*$.
This shows that establishing $L_p$ theory for equations as in \eqref{intro_maineq} can be highly beneficial in dealing with the usual parabolic equations.
For the appearance of the term $u_t I_{\Omega^*}$ on the right-hand side of parabolic equations, see \cite[Proof of Proposition 5.1]{MR4387198} and \cite[Proof of Corollary 4.7]{MR2139880}, where $u_t$ is sufficiently smooth because the coefficients are constant.
However, in general, especially when $a_{ij}$ are merely measurable in time, one cannot expect $u_t$ to belong to any Lebesgue spaces.

We now comment on the conditions for the coefficients $a_{ij}$ in this paper.
It is worth noting that when $p=2$, as in \cite{MR0200593}, our main results (Theorems \ref{main whole space}, \ref{thm0521_1}, and also see Theorem \ref{Interior L_2 solvability}) hold without any regularity assumptions (under the ellipticity condition \eqref{eq0521_02}) on the leading coefficients $a_{ij}$.
However, for $p \neq 2$, as shown in \cite{MR3488249}, there exists a counterexample to the $L_p$ theory for the usual parabolic equations as in \eqref{intro_classiceqn}.
In fact, it is not possible for all of the coefficients $a_{ij}(t,x)$ to be merely measurable in both time and one spatial variable, i.e., to have no regularity conditions as functions of time and one spatial variable.
Thus, in this paper, we allow $a_{ij}(t,x)$ to be merely measurable in time (Theorem \ref{main whole space}) or to be merely measurable in one spatial variable, in particular, $x_1$ (Theorem \ref{thm0521_1}).
As functions of the remaining variables, we impose that $a_{ij}$ have small mean oscillations (see Assumptions \ref{assum coeffi} and \ref{assum0521_1}).
Indeed, for the usual parabolic equations \eqref{intro_classiceqn} it is possible to allow all $a_{ij}$ but one (for instance, $a_{11}$) to be merely measurable both in time and $x_1$ (see, for instance, \cite{MR2764911}), but in this paper, we do not pursue this direction.

As for the domain where \eqref{intro_maineq} is defined, we have $\bR \times \bR^d = \{(t,x): t \in \bR,\, x = (x_1,\ldots,x_d) \in \bR^d\}$.
For domains such as $\bR \times \Omega$, where $\Omega \subset \bR^d$, if $\partial\Omega$ and $a_{ij}$ are sufficiently smooth (for example, if $a_{ij}(t,x)$ have small mean oscillations as functions of $(t,x)$), the results (with the Dirichlet or conormal derivative boundary conditions) corresponding to the main results of this paper can be derived by the even or odd extension technique with Theorem \ref{thm0521_1}, where $a_{ij}$ are allowed to be merely measurable in one spatial direction.
For the even or odd extension technique, see, for instance, \cite{MR2764911}.
For domains having a finite time interval like $(0,T) \times \Omega$, because \eqref{intro_maineq} involves a non-local time derivative, we do not consider such domains in this paper, but they will be dealt with elsewhere.

We now briefly discuss how to prove our main results (Theorems \ref{main whole space} and \ref{thm0521_1}).
    Based on the mean oscillation argument in \cite{MR2304157}, our proof proceeds by first decomposing the solution $u$ of the equation locally into its homogeneous part and inhomogeneous part. It is critical to observe that, due to the non-local operator $D_t^{1/2}$ on the right-hand side, the equation must be analyzed on $\bR\times B_1$ rather than the standard cylinder $Q_1=(-1,0) \times B_1 = \{(t,x) \in \bR \times \bR^d: -1<t<0, |x| < 1\}$.
Thus, obtaining a local estimate that takes into account the non-local term on the right-hand side is crucial (see Lemma \ref{local L_2 estimate lemma}).
    After that, we derive a Lipshitz estimate for the homogeneous solution $v$ and $Dv$ to obtain the mean oscillation estimates of $u$ and $Du$.

    To derive the mean oscillation estimate of $D_t^{1/2} u$, we temporarily assume that $u$ is a solution of the heat equation, which means $D_i(a_{ij}D_j u)=\Delta u$. Then, we exploit the fact that $D_t^{1/2}v $ also serves as a solution to the equation for the homogeneous solution $v$ to get the mean oscillation estimate of $D_t^{1/2} u$. Finally, by appropriately manipulating the equation (see \eqref{manipulated eqn}), we obtain the desired estimates.
    
    The remaining part of the paper is organized as follows. We introduce some notation and state the main results in the next section. Section \ref{Auxiliary results} contains some auxiliary results.
In Section \ref{sec L2 estimates}, we establish an $L_2$-estimate and a local $L_2$-estimate.
Then, in Sections \ref{sec05} and \ref{sec06}, we obtain mean oscillation estimates when the coefficients $a_{ij}$ satisfy Assumption \ref{assum coeffi} and prove Theorem \ref{main whole space}. Finally, we prove Theorem \ref{thm0521_1} in Section \ref{sec x_1}.

\section{notation and main results}							\label{sec01}

Let $d\ge 1$ be a positive integer, and let $\Omega$ be a domain in $\bR^d$, where $\bR^d$ denotes the $d$-dimensional Euclidean space. A point in $\bR^d$ is written as $x = (x_1,x_2,\ldots,x_d) = (x_1,x')$, where $x' \in \bR^{d-1}$.
We write
\[
B_r(x)=\{y\in \bR^d:|x-y|<r\}, \quad B_r=B_r(0).
\]
For $X=(t,x)\in \bR^{d+1}$, set
\[
Q_{r,s}(X)=(t-r^2,t+r^2)\times B_s(x), 
\quad Q_r(X)=Q_{r,r}(X), 
 \quad Q_r=Q_r(0). 
\]
When $d \geq 2$, by $B_r'(x')$, $x' \in \bR^{d-1}$, we mean a ball with radius $r$ in $\bR^{d-1}$, i.e.,
\[
B_r'(x') = \{y' \in \bR^{d-1}: |x'-y'| < r\}.
\]
For parabolic cylinders with respect to $(t,x')$-variable, we use the notation
\[
Q'_r(t,x')=(t-r^2,t+r^2)\times B'_r(x') \subset \bR \times \bR^{d-1}.
\]
Let $\cD \subset \bR^{d+1}$.
We write $(u)_\cD$ to denote
\begin{equation}
							\label{eq0521_01}
(u)_\cD = \dashint_{\cD} u(t,x) \, dx \, dt = \dashint_{\cD} u(t,x) \, dX = \frac{1}{|\cD|}\int_{\cD} u(t,x) \, dX,
\end{equation}
where $|\cD|$ is the $(d+1)$-dimensional Lebesgue measure of $\cD$.
As above, we often use $dX$ to denote $dx\,dt$ or $dt \, dx$.
In a similar manner, we define the same averaging notation for integrals with respect to $x \in\bR^d$ or $(t,x') \in \bR \times \bR^{d-1}$.

Throughout the paper, unless explicitly stated, we set $\cQ = \bR \times \Omega$.
By $C^\infty(\cQ)$ ($C_0^\infty(\cQ)$) we refer to the collection of infinitely differentiable functions in $\cQ$ (with compact support in $\cQ$).

For $u \in L_p(\cQ)$, $1 < p < \infty$, let $H (u)$ be the Hilbert transform of $u$ with respect to time.
If $u \in C_0^\infty(\cQ)$, one can write
\[
H(u)(t,x) = \lim_{\varepsilon \to 0} H^{(\varepsilon)} (u)(t,x) = \lim_{\varepsilon \to 0} \frac{1}{\pi} \int_{|s| \geq \varepsilon} \frac{u(t-s,x)}{s} \, ds.
\]
Moreover, if $u \in C_0^\infty(\cQ)$, for each $x \in \Omega$, we have (see \cite[Theorem 5.1.7]{MR3243734})
\[
\int_\bR |H (u)(t,x)|^p \, dt \leq N \int_\bR |u(t,x)|^p \, dt,
\]
where $N = N(p)$, from which we readily have
\begin{equation}
							\label{eq0510_01}
\|H(u)\|_{L_p(\cQ)} \leq N \|u\|_{L_p(\cQ)}.
\end{equation}
Hence, as an extension, the Hilbert transform $H$ is well-defined and bounded as an operator on $L_p(\cQ)$.

For $u \in L_1(\cQ)$, we define the Fourier transform of $u$ with respect to time as
\[
\tilde{u}(\xi,x) = \frac{1}{\sqrt{2\pi}} \int_\bR e^{-i \xi t} u(t,x) \, dt.
\]
For $u(t,x) \in L_2(\cQ)$, we define the Fourier transform of $u$ with respect to $t \in \bR$ by using an approximating sequence in $C_0^\infty(\cQ)$ converging to $u$ in $L_2(\cQ)$.
We also define inverse Fourier transforms with respect to time in an obvious manner.
For $\xi \in \bR$, we denote
\[
\operatorname{sgn}(\xi)=\begin{cases}
    -1 \quad &\text{if} \quad \xi<0, 
    \\
    0 \quad &\text{if} \quad \xi=0, 
    \\
    1 \quad &\text{if} \quad \xi>0.
\end{cases}
\]
Then, we see that
\begin{equation}
							\label{eq0519_01}
\widetilde{H(\varphi)}(\xi,x) = - i \operatorname{sgn}(\xi) \tilde{\varphi}(\xi,x), \quad \varphi \in C_0^\infty(\cQ).
\end{equation}

For $\varphi \in C_0^\infty(\cQ)$ (or for sufficiently smooth $\varphi$ such that the integral below is finite), we define the half-time derivative of $\varphi$ with respect to time as
\begin{equation}
    \label{closed form}
D_t^{1/2} \varphi{(t,x)} = -(-\Delta_t)^{1/4}\varphi(t,x) = \frac{1}{\sqrt{8\pi}} \int_\bR \frac{\varphi(t+\ell,x)-\varphi(t,x)}{|\ell|^{3/2}} \, d\ell,
\end{equation}
where $\Delta_t = \partial^2/ \partial t^2$.
One can check that the Fourier transform of $D_t^{1/2}\varphi$ with respect to time is
\begin{equation}
							\label{eq0519_02}
\widetilde{D_t^{1/2}\varphi}(\xi,x) = - |\xi|^{1/2}\widetilde{\varphi}(\xi,x).
\end{equation}
Note that $D_t^{1/2}$ can also be called the fractional Laplacian of order $1/4$ with respect to time.

On the other hand, as usual, $Du$ denotes the spatial gradient of $u$, given by
\[
Du=(D_1u,D_2u,\dots, D_d u)=(D_{x_1}u,D_{x_2}u,\dots, D_{x_d}u)
\]
or $Du$ denotes one of its components $Du = D_ju = D_{x_j}u$.
\begin{remark}
							\label{rem0515_1}
For $\varphi \in C_0^\infty(\cQ)$, it is easy to see that $D_t^{1/2}\varphi \in L_p(\cQ)$ for any $p \in [1,\infty]$.
Moreover, $D_t^{1/2} \varphi \in C^\infty(\cQ)$ and $
\partial_t (D_t^{1/2} \varphi) = D_t^{1/2}(\varphi_t)$ and $D D_t^{1/2}\varphi = D_t^{1/2} D \varphi$.
\end{remark}

We now define $D_t^{1/2} u$ for functions with half-time derivatives in $L_p(\cQ)$.

\begin{definition}\label{weak form}
Let $1 < p < \infty$.
For $u \in L_p(\cQ)$, if there exists $v \in L_p(\cQ)$ satisfying
\begin{equation}
							\label{eq0513_03}
\int_\cQ u \, D_t^{1/2} \varphi \, dX = \int_\cQ v \, \varphi \, dX
\end{equation}
for all $\varphi \in C_0^\infty(\cQ)$, then we write
\[
v = D_t^{1/2}u.
\]
\end{definition}

 For $p\in (1,\infty)$, we define
\[
H_p^{1/2,1}(\cQ)= \{u:u,Du,D_t^{1/2}u\in L_p(\cQ)\},
\]
equipped with the norm 
\[
\|u\|_{H_p^{1/2,1}(\cQ)}=\|D_t^{1/2}u\|_{L_p(\cQ)}+\|Du\|_{L_p(\cQ)}+\|u\|_{L_p(\cQ)}.
\]
Let $\mathring H_p^{1/2,1}(\cQ)$ be the closure of $C_0^\infty(\cQ)$ with respect to the norm of $H_p^{1/2,1}(\cQ)$.
Let $-\infty\le S<T\le \infty$ and $m\in \mathbb{N}$. For a multi-index $\alpha=(\alpha_1,\dots,\alpha_d)$, we use the notation $D^\alpha u=D_1^{\alpha_1}\dots D_d^{\alpha_d} u$.
\[
W_p^{1,m}\left((S,T)\times \Omega\right)=\{u: u,\partial_t u, D^\alpha u\in L_p\left((S,T)\times \Omega\right),\,\,  \text{for} \,\, |\alpha|\le m\},
\]
\[W_p^{1,\infty}\left((S,T)\times \Omega\right)= \bigcap_{m=1}^\infty W_p^{1,m}\left((S,T)\times \Omega\right).
\]

Now, we introduce equations along with the definition of solutions.
For $1\le i,j\le d$, let  $a_{ij}$, $b_i$, $c$ be bounded measurable functions on $\cQ$. 
Let $f\in L_p(\cQ)$, $h\in L_p(\cQ)$, and $g =(g_1, \ldots, g_d) \in \left(L_p(\cQ)\right)^d$.
We say that $u\in H_p^{1/2,1}(\cQ)$ is a solution to the equation 
\[
u_t-D_i(a_{ij}D_j u)+b_i D_iu+c u =D_t^{1/2}h+D_ig_i+f \quad \text{in} \,\, \cQ
\]
if  for any $\varphi \in C_0^\infty(\cQ)$, we have 
\begin{multline}
							\label{eq0523_04}
\int_\cQ- H(D_t^{1/2}u) \, D_t^{1/2}\varphi +a_{ij}D_juD_i\varphi+b_i D_iu \, \varphi+ cu\varphi\,dX
\\
=\int_\cQ h \, D_t^{1/2}\varphi
-g_iD_i\varphi+ f\varphi\,dX.
\end{multline}
We remark that (see Lemma \ref{lem0514_1})
\[
\int_\cQ H(D_t^{1/2}u) \, D_t^{1/2}\varphi\,dX= \int_\cQ u \varphi_t\,dX.
\]
Therefore, the above definition of solutions is consistent with the following classical definition of solutions:  
$u\in W_p^{0,1}((S,T)\times \Omega)$ is a solution to the equation 
\begin{equation}
							\label{eq0709_01}
u_t-D_i(a_{ij}D_j u)+b_i D_iu+c u =D_ig_i+f \quad \text{in} \,\, (S,T)\times \Omega
\end{equation}
if, for any $\varphi \in C_0^\infty( [S,T]\times \Omega )$ vanishing at $t=S$ and $T$, we have 
\[
\int_{(S,T)\times \Omega} - u \varphi_t+a_{ij}D_juD_i\varphi+b_i D_iu \varphi+ cu\varphi \,dX
\]
\[
=\int_{(S,T)\times \Omega}-g_iD_i\varphi
+ f\varphi\,dX.
\]

Throughout the paper, we impose the strong ellipticity condition and boundedness on the coefficients $a_{ij}$ with $\delta \in (0,1]$. More precisely, we assume that
\begin{equation}
    \label{eq0521_02}
\delta|\xi|^2 \le a_{ij}(t,x)\xi_i\xi_j, \quad |a_{ij}(t,x)|\le \delta^{-1}, 
\end{equation}
for all $\xi\in \bR^d$, and $(t,x)\in \bR^{d+1}$.

The following are our regularity assumptions for the coefficients $a_{ij}$.

\begin{assumption}[$\gamma$]                                        \label{assum coeffi}
There is a constant $R_0\in (0,1]$ such that, for each $(t,x) \in \bR^{d+1}$ and $r \in (0,R_0]$,
\[
\dashint_{Q_r(t,x)} |a_{ij}(s,y) - \bar{a}_{ij}(s)| \, dy \, ds \leq \gamma,
\]
where
\[
\bar{a}_{ij}(s) = \dashint_{B_r(x)} a_{ij}(s,z) \, dz.
\]
\end{assumption}

\begin{remark}
							\label{rem0522_01}
If $a_{ij}=a_{ij}(t)$, then Assumption \ref{assum coeffi} is satisfied regardless of the regularity of $a_{ij}$.
Assumption \ref{assum coeffi} ($\gamma$) implies that, for any $x \in \bR^d$ and any $a,b, r \in \bR$ such that $b-a \geq 2r^2$ and $r \in (0, R_0]$,
\[
\frac{1}{b-a}\int_a^b \dashint_{B_r(x)} |a_{ij}(s,y) - a_{ij}(s)| \, dy \, ds \leq 2 \gamma.
\]
See \cite[Remark 2.3]{MR3899965} and Remark \ref{rem0522_02} below.
\end{remark}

\begin{assumption}[$\gamma$]                                        \label{assum0521_1}
There is a constant $R_0\in (0,1]$ such that, for each $(t,x) = (t,x_1,x') \in \bR^{d+1}$ and $r \in (0,R_0]$,
\[
\dashint_{Q_r(t,x)} |a_{ij}(s,y_1,y') - \bar{a}_{ij}(y_1)| \, dy \, ds \leq \gamma,
\]
where
\[
\bar{a}_{ij}(y_1) = \dashint_{Q'_r(t,x')} a_{ij}(\tau,y_1,z') \, dz' \, d \tau.
\]
If $d=1$, we read $Q'_r(t,x') = (t-r^2,t+r^2)$.
\end{assumption}

\begin{remark}
							\label{rem0522_02}
If $a_{ij} = a_{ij}(x_1)$, then Assumption \ref{assum0521_1} is satisfied without any regularity conditions on $a_{ij}$.
For any $x \in \bR^d$ and any $a,b,r \in \bR$ such that $b-a \geq 2r^2$, $r \in (0,R_0]$, find a positive integer $k$ satisfying
\[
b - 2(k+1)r^2 < a \leq b - 2k r^2, \quad \text{i.e.}, \quad \frac{1}{k+1} < \frac{2r^2}{b-a} \leq \frac{1}{k}.
\]
Then, we set
\[
b_\ell = b - 2\ell r^2 - r^2, \quad \ell = 0, 1, \ldots, k.
\]
We now consider coefficients $\bar{a}_{ij}$ defined by
\[
\bar{a}_{ij}(s,y_1) = \bar{a}^\ell_{ij}(y_1) \quad \text{for} \,\, s \in \left(b_\ell-r^2, b_\ell+r^2\right],
\]
where
\[
\bar{a}_{ij}^\ell(y_1) = \dashint_{Q_r'(b_\ell,x')}a_{ij}(\tau,y_1,z') \, dz' \, d\tau.
\]
Then, by Assumption \ref{assum0521_1} 
\[
\left( |a_{ij} - \bar{a}_{ij}| \right)_{(a,b) \times B_r(x)}
= \frac{1}{b-a} \int_a^b \dashint_{B_r(x)} |a_{ij}(s,y) - \bar{a}_{ij}(s,y_1) | \, dy \, ds
\]
\[
\leq \frac{1}{b-a} \sum_{\ell=0}^k \int_{b_\ell - r^2}^{b_\ell+r^2} \dashint_{B_r(x)} |a_{ij}(s,y) - \bar{a}_{ij}^\ell(y_1)| \, dy \, ds
\]
\[
= \frac{2r^2}{b-a} \sum_{\ell=0}^k \dashint_{Q_r(b_\ell,x)} \left|a_{ij}(s,y_1,y') - \dashint_{Q_r'(b_\ell,x')} a_{ij}(\tau,y_1,z') \, dz' \, d\tau \right| \, dy \, ds
\]
\[
\leq \frac{2r^2}{b-a} (k+1) \gamma \leq \frac{k+1}{k} \gamma \leq 2 \gamma.
\]
\end{remark}

We are now ready to introduce our main results. 

\begin{theorem}
    \label{main whole space}
Let $p\in(1,\infty)$, $\lambda \geq 0$, $h, g_i, f \in L_p(\bR^{d+1})$, $i=1,\ldots,d$, with $f\equiv 0$ if $\lambda = 0$.
Then, there exist positive constants $\gamma=\gamma(d,\delta,p)$, and $N=N(d,\delta,p)$ such that, under Assumption \ref{assum coeffi} ($\gamma$), for any $u\in H_p^{1/2,1}(\bR^{d+1})$ satisfying 
        \begin{equation}
        \label{main eq}
    u_t-D_i(a_{ij}D_j u)+\lambda u=D_t^{1/2}h+D_ig_i+f \quad \text{in} \,\, \bR^{d+1},
    \end{equation}
we have
\begin{multline}
        \label{main whole space estimate}
\|D_t^{1/2}u\|_{L_p(\bR^{d+1})} + \|Du\|_{L_p(\bR^{d+1})} + \sqrt{\lambda}\|u\|_{L_p(\bR^{d+1})}
\\
\le N\left( \|h\|_{L_p(\bR^{d+1})} + \|g_i\|_{L_p(\bR^{d+1})} + \lambda^{-1/2}\|f\|_{L_p(\bR^{d+1})}\right),
\end{multline}
provided that $\lambda \ge \lambda_0$, where $\lambda_0=\lambda_0(d,\delta,p,R_0)\geq 0$.
Moreover, for $\lambda > \lambda_0$, there exists a unique solution $u\in H_p^{1/2,1}(\bR^{d+1})$ satisfying the equation \eqref{main eq}.
\end{theorem}

\begin{remark}
In Theorem \ref{main whole space} as well as Theorem \ref{thm0521_1}, one may also consider equations with lower-order terms as in \eqref{eq0709_01} with appropriate conditions on $b_i$ and $c$, such as boundedness.
For simplicity, in this paper we consider the case where $b_i = c = 0$.
\end{remark}

\begin{theorem}
    \label{thm0521_1}
The assertions in Theorem \ref{main whole space} hold under Assumption \ref{assum0521_1} ($\gamma$).
\end{theorem}

\section{Auxiliary results}\label{Auxiliary results}

As seen in \eqref{eq0510_01}, the Hilbert transform $H$ is a bounded operator on $L_p(\cQ)$.
Moreover, $H^{(\varepsilon)}(u)(t,x) \to H(u)(t,x)$ in $L_p(\cQ)$ and a.e. in $\cQ$ as $\varepsilon \to 0$, where we recall that
\[
H^{(\varepsilon)}(u)(t,x) = \frac{1}{\pi} \int_{|s| \geq \varepsilon} \frac{u(t-s,x)}{s} \, ds.
\]
That is, we have

\begin{lemma}\label{Hilbert pp bound}
For $p\in (1,\infty)$, there is $N=N(p)$ such that  
\[
\|H (u)\|_{L_p(\cQ)}\le N \|u\|_{L_p(\cQ)}, 
\]
for all $u \in L_p(\cQ)$.
In particular, for $p=2$, the constant $N=1$.
Moreover,
\[
H^{(\varepsilon)}(u) \to H (u)
\]
in $L_p(\cQ)$ and almost everywhere in $\cQ$ as $\varepsilon \to 0$.
\end{lemma}

\begin{proof}
The lemma follows from the results, for instance, in \cite[Chapter 5]{MR3243734}, regarding the operators $H$ and $H^{(\varepsilon)}$.
\end{proof}

If $u \in L_p(\cQ)$ with $D_t^{1/2}u \in L_p(\cQ)$, then $u$ can be approximated by a sequence in the same class having compact support with respect to time.
In fact, one can obtain an approximating sequence by multiplying an appropriate cut-off function in time.
See Lemma \ref{tail estimate} below.
Before this lemma, we first observe the following, which can be proved by the definition of $D_t^{1/2}u$ and Fubini's theorem.

\begin{lemma}
							\label{lem0513_1}
Let $p \in (1,\infty)$.
If $u \in L_p(\cQ)$ with $D_t^{1/2}u \in L_p(\cQ)$ and $\eta \in C_0^\infty(\bR)$, 
then $D_t^{1/2}(\eta u) \in L_p(\cQ)$ and
\[
D_t^{1/2} (\eta u) = \eta D_t^{1/2}u + \frac{1}{\sqrt{8\pi}} \int_\bR u(t+\ell,x)\frac{\eta(t+\ell)-\eta(t)}{|\ell|^{3/2}} \, d\ell.
\]
\end{lemma}

In the lemma below, we use the following sequence of functions in time.
For $k=0,1,2,\dots,$ let $0\le \eta_k\le 1$ be a $C^\infty(\bR)$ function such that 
\begin{equation}
							\label{eq0513_02}
\eta_k(t)
=\begin{cases}
1 \quad &\text{for} \,\, t\in (-2^k,2^k),
\\
0 \quad &\text{for} \,\, t\in \bR\setminus (-2^{k+1},2^{k+1}),
\end{cases}
\end{equation}
and 
\begin{equation}
    \label{cut off bound}
    |(\eta_k)_t|\le N2^{-k},
\end{equation}
where $N$ is a positive constant independent of $k$.

\begin{lemma}
    \label{tail estimate}
Let $p\in(1,\infty)$.
For a function $u$ defined on $\cQ$ such that the right-hand side of \eqref{eq0513_01} is finite,
set
\[
u_k(t,x) = \frac{1}{\sqrt{8\pi}} \int_\bR u(t+\ell,x)\frac{\eta_k(t+\ell)-\eta_k(t)}{|\ell|^{3/2}} \, d\ell.
\]
Then,
\begin{equation}
							\label{eq0513_01}
\|u_k\|_{L_p(\cQ)} \le N2^{-k/2} \sum_{j=1}^{\infty}2^{-j(1/2+1/p)}\|u\|_{L_p\left((-2^{k+j},2^{k+j}) \times \Omega\right)},
\end{equation}
where $N = N(p)$.
Moreover, if $u \in L_p(\cQ)$ and $D_t^{1/2}u \in L_p(\cQ)$, then
\begin{equation}
							\label{eq0520_02}
u_k = D_t^{1/2}(u\eta_k) - \eta_k D_t^{1/2}u,
\end{equation}
so that the estimate \eqref{eq0513_01} holds with $u_k$ replaced with $D_t^{1/2}(u\eta_k) - \eta_k D_t^{1/2}u$.
\end{lemma}

\begin{proof}
By Lemma \ref{lem0513_1}, it is clear that \eqref{eq0520_02} holds if $u, D_t^{1/2}u \in L_p(\cQ)$.
To prove \eqref{eq0513_01}, we write
\[
\sqrt{8\pi}|u_k| \leq \int_{\bR} |u(t+\ell,x)|\frac{|\eta_k(t+\ell)-\eta_k(t)|}{|\ell|^{3/2}} \, d\ell
\]
\[
=\int_{-2^{k+3}}^{2^{k+3}}\cdots+\int_{-\infty}^{-2^{k+3}}\cdots +\int_{2^{k+3}}^{\infty}\cdots
=:J_1+J_2+J_3.
\]
By \eqref{cut off bound}, we have 
\begin{equation*}
    |\eta_k(t+\ell)-\eta_k(t)|\le N2^{-k}|\ell|.
\end{equation*}
We then see that 
\[
J_1(t,x) = \int_{-2^{k+3}}^{2^{k+3}}|u(t+\ell,x)||\eta_k(t+\ell)-\eta_k(t)|\frac{d\ell}{|\ell|^{3/2}}
\]
\[
\le N2^{-k}\int_{-2^{k+3}}^{2^{k+3}}|u(t+\ell,x)|\frac{d\ell}{|\ell|^{1/2}},
\]
which, along with Minkowski's inequality, shows
\begin{multline}
    \label{local 14}
\|J_1\|_{L_p\left((-2^{k+2},2^{k+2}) \times \Omega \right)}
\\
\le N2^{-k}\int_{-2^{k+3}}^{2^{k+3}}\|u(\cdot+\ell,\cdot)\|_{L_p\left((-2^{k+2},2^{k+2}) \times \Omega\right)}\, \frac{d\ell}{|\ell|^{1/2}}.
\end{multline}
Note that, for $\ell\in (-2^{k+3},2^{k+3})$, 
\[
\|u(\cdot+\ell,\cdot)\|_{L_p\left((-2^{k+2},2^{k+2}) \times \Omega \right)}^p
=
\int_{-2^{k+2}}^{2^{k+2}} \int_\Omega |u(t+\ell,x)|^p \, dx \,dt
\]
\[
\le
\int_{-2^{k+4}}^{2^{k+4}} \int_\Omega |u(t,x)|^p \, dx \,dt
=
\|u\|_{L_p\left((-2^{k+4},2^{k+4}) \times \Omega \right)}^p.
\]
From this inequality and \eqref{local 14}, we have
\[
\|J_1\|_{L_p\left((-2^{k+2},2^{k+2}) \times \Omega\right)}
\le
N2^{-k}\|u\|_{L_p\left((-2^{k+4},2^{k+4}) \times \Omega\right)}\int_{-2^{k+3}}^{2^{k+3}}\frac{d\ell}{|\ell|^{1/2}}
\]
\begin{equation}
    \label{local 15}
    \le N2^{-k/2}\|u\|_{L_p\left((-2^{k+4},2^{k+4}) \times \Omega \right)}. 
\end{equation}

We now estimate $\|J_2\|_{L_p\left((-2^{k+2},2^{k+2}) \times \Omega \right)}$. 
For $t\in (-2^{k+2},2^{k+2})$,
we obtain that
\[
J_2(t,x)
= \int_{-\infty}^{-2^{k+3}}|u(t+\ell,x)||\eta_k(t+\ell)-\eta_k(t)| \, \frac{d\ell}{|\ell|^{3/2}}
\]
\[
\le 2 \int_{-\infty}^{-2^{k+3}}|u(t+\ell,x)|\, \frac{d\ell}{|\ell|^{3/2}} = 2 \sum_{j=0}^{\infty} \int_{-2^{k+4+j}}^{-2^{k+3+j}}|u(t+\ell,x)|\, \frac{d\ell}{|\ell|^{3/2}}
\]
\[
\le N \sum_{j=0}^{\infty}2^{-3(k+j)/2}  \int_{-2^{k+4+j}}^{-2^{k+3+j}}|u(t+\ell,x)| \, d\ell
\]
\[
\le 
N \sum_{j=0}^{\infty}2^{-3(k+j)/2}\int_{-2^{k+5+j}}^{-2^{k+2}}|u(\ell,x)| \, d\ell
\]
\[
\le 
N2^{-k(1/2+1/p)}
\sum_{j=0}^{\infty}2^{-j(1/2+1/p)}\|u(\cdot,x)\|_{L_p((-2^{k+2},2^{k+5+j}))},
\]
where we apply H\"older's inequality with respect to $t$ to get the last inequality.
From this estimate, it follows that 
\begin{multline}
    \label{local 16}
\|J_2\|_{L_p\left((-2^{k+2},2^{k+2}) \times \Omega\right)}
\\
\le N2^{-k/2}
\sum_{j=0}^{\infty}2^{-j(1/2+1/p)}\|u\|_{L_p\left((-2^{k+5+j},-2^{k+2}) \times \Omega \right)}.
\end{multline}
In the same manner, we have 
\begin{multline}
    \label{local 17}
\|J_3\|_{L_p\left((-2^{k+2},2^{k+2}) \times \Omega \right)}
\\
\le N2^{-k/2}
\sum_{j=0}^{\infty}2^{-j(1/2+1/p)}\|u\|_{L_p\left((2^{k+2},2^{k+5+j})\times \Omega \right)}.
\end{multline}
Combining the estimates \eqref{local 15},\eqref{local 16}, and \eqref{local 17}, we reach that 
\begin{multline}
    \label{local 18}
\|u_k\|_{L_p\left((-2^{k+2},2^{k+2}) \times \Omega \right)}
\\
\le N2^{-k/2}
\sum_{j=0}^{\infty}2^{-j(1/2+1/p)}\|u\|_{L_p\left((-2^{k+j},2^{k+j}) \times \Omega \right)}.
\end{multline}

To estimate $\|u_k\|_{L_2\left((-\infty,-2^{k+2}) \times \Omega \right)}$,
we notice that, for $t\in(-\infty,-2^{k+2})$,
\[
\sqrt{8\pi}|u_k(t,x)| \leq \int_{-\infty}^{\infty}|u(t+\ell,x)||\eta_k(t+\ell)-\eta_k(t)| \, \frac{d\ell}{|\ell|^{3/2}}
\]
\[
=
\int_{-\infty}^{\infty}|u(t+\ell,x)||\eta_k(t+\ell)| \,\frac{d\ell}{|\ell|^{3/2}} = \int_{-\infty}^\infty |u(\ell,x)||\eta_k(\ell)| \, \frac{d\ell}{|\ell-t|^{3/2}}
\]
\[
\leq \int_{-2^{k+1}}^{2^{k+1}}|u(\ell,x)| \, \frac{d\ell}{|\ell-t|^{3/2}}
\]
\[
\le
\left(\int_{-2^{k+1}}^{2^{k+1}}|u(\ell,x)|^p\,d \ell\right)^{1/p}\left(\int_{-2^{k+1}}^{2^{k+1}}|\ell-t|^{-3q/2}\,d\ell\right)^{1/q},
\]
where the last inequality comes from H\"older's inequality with $1/p + 1/q = 1$.
By direct computation, for $t\in (-\infty,-2^{k+2})$,
 we see that 
\[
\int_{-2^{k+1}}^{2^{k+1}}|\ell-t|^{-3q/2}\,d\ell 
\le \int_{-2^{k+1}}^{2^{k+1}}|-2^{k+1}-t|^{-3q/2}\,d\ell
\]
\[
= N2^{k}|t+2^{k+1}|^{-3q/2}.
\]
From the above estimates, we get
\[
\|u_k\|_{L_p((-\infty,-2^{k+2}))}^p
\]
\[
\le N2^{kp/q}
\int_{-\infty}^{-2^{k+2}}
\left(\int_{-2^{k+1}}^{2^{k+1}}|u(\ell,x)|^p\,d\ell \, dx\right)|t+2^{k+1}|^{-3p/2}\,dt
\]
\[
= N2^{kp/q}
 \|u\|_{L_p\left((-2^{k+1},2^{k+1}) \times \Omega \right)}^p
\int_{-\infty}^{-2^{k+2}}|t+2^{k+1}|^{-3p/2}\,dt
\]
\[
= N2^{kp/q}
\|u\|_{L_p\left((-2^{k+1},2^{k+1}) \times \Omega \right)}^p
\int_{2^{k+1}}^{\infty}|t|^{-3p/2}\,dt
\]
\[
\le N2^{-kp/2}\|u\|_{L_p\left((-2^{k+1},2^{k+1}) \times \Omega\right)}^p.
\]
That is,
\begin{equation}
    \label{local 19}
\|u_k\|_{L_p\left((-\infty,-2^{k+2}) \times \Omega \right)}
\le
N2^{-k/2}\|u\|_{L_p\left((-2^{k+1},2^{k+1})\times \Omega\right)}.
\end{equation}
Analogously, we have
\begin{equation}
    \label{local 20}
\|u_k\|_{L_p\left((2^{k+2},\infty) \times \Omega \right)}
\le
N2^{-k/2}\|u\|_{L_p\left((-2^{k+1},2^{k+1}) \times 
\Omega \right)}.
\end{equation}
Upon combining \eqref{local 18},\eqref{local 19}, and \eqref{local 20}, we conclude the inequality \eqref{eq0513_01}.
\end{proof}

We now observe that whenever $D_t^{1/2}u \in L_p(\cQ)$,  there exists an approximating sequence in the class $C_0^\infty(\cQ)$.

\begin{lemma}
							\label{lem0509_2}
Let $1 < p < \infty$.
For $u \in L_p(\cQ)$ such that $D_t^{1/2}u \in L_p(\cQ)$, there exists a sequence $u^k \in C_0^\infty(\cQ)$ such that
\[
u^k \to u, \quad D_t^{1/2}u^k \to D_t^{1/2}u
\]
in $L_p(\cQ)$.
\end{lemma}

\begin{proof}
We give a sketched proof.
Since $u$ is free from spatial derivatives, one can assume that $u$ has compact support in $x \in \Omega$.
In addition, Lemma \ref{tail estimate} implies that $u$ can be approximated by functions with compact support in $t \in \bR$.
Thus, it is enough to prove the lemma for $u$ having compact support in $\cQ$.
In this case, we take the mollification $u^{(\varepsilon)}$ of $u$ to generate $u^k$, for which one can check that
\[
D_t^{1/2}(u^{(\varepsilon)}) = (D_t^{1/2}u)^{(\varepsilon)}.
\]
This, in particular, proves the convergence of $D_t^{1/2}u^k \to D_t^{1/2}u$ in $L_p(\cQ)$.
\end{proof}

\begin{remark}
							\label{rem0506_1}
Thanks to Lemma \ref{lem0509_2}, the equality \eqref{eq0513_03} holds for all $\varphi \in L_q(\cQ)$, provided that $D_t^{1/2}\varphi \in L_q(\cQ)$, $1/p+1/q=1$.
\end{remark}

We now summarize a few properties necessary for the remaining part of this paper regarding $H$ and $D_t^{1/2}$.
In particular,
\begin{equation*}
D_t^{1/2}H (D_t^{1/2} \varphi) = - \varphi_t
\end{equation*}
for $\varphi \in C_0^\infty(\cQ)$.

\begin{lemma}
							\label{lem0514_1}
Let $1< p < \infty$.
\begin{enumerate}
\item For $u \in L_2(\cQ)$,
\[
\widetilde{H(u)}(\xi,x) = - i \operatorname{sgn}(\xi) \tilde{u}(\xi,x).
\]

\item For $u \in L_p(\cQ)$, $H(H(u)) = H^2(u) = - u$.

\item For $u \in L_2(\cQ)$ with $D_t^{1/2} u \in L_2(\cQ)$, we have
\[
\widetilde{D_t^{1/2} u}(\xi,x) = -|\xi|^{1/2} \tilde{u}(\xi,x).
\]

\item For $u \in L_p(\cQ)$ with $D_t^{1/2}u \in L_p(\cQ)$,
\[
\int_\cQ H (D_t^{1/2} u) \, D_t^{1/2} \varphi \, dX = \int_\cQ u \, \varphi_t \, dX
\]
for all $\varphi \in C_0^\infty(\cQ)$, and
\[
H (D_t^{1/2} u) = D_t^{1/2}H(u).
\]

\item For $u \in L_p(\cQ)$ with $Du \in L_p(\cQ)$,
\begin{equation}
							\label{eq0514_01}
D_{x_j} H(u) = H(D_{x_j}u).\end{equation}

\item For $u \in L_p(\cQ)$ such that $D_t^{1/2}u, Du \in L_p(\cQ)$,
\begin{equation}
							\label{eq0515_02}
D_t^{1/2}(D_{x_j}u) = D_{x_j}(D_t^{1/2}u).
\end{equation}
\end{enumerate}
\end{lemma}

\begin{proof}
These assertions follow from the definitions of $H$ and $D_t^{1/2}$, along with \eqref{eq0519_01}, \eqref{eq0519_02}, and Lemmas \ref{Hilbert pp bound} and \ref{lem0509_2}.
In particular, if $u \in C_0^\infty(\cQ)$, then
\[
\int_\cQ H(D_t^{1/2}u) D_t^{1/2}\varphi \, dX = -\int_\Omega \int_\bR i \operatorname{sgn}(\xi) |\xi|^{1/2} \widetilde{u}(\xi,x) |\xi|^{1/2} \overline{\widetilde{\varphi}(\xi,x)} \, d\xi \, dx
\]
\[
= \int_\cQ \widetilde{u}(\xi,x) \, \overline{i \xi \widetilde{\varphi}(\xi,x)} \, d\xi \, dx = \int_\cQ u \, \varphi_t \, dX.
\]
Also see \cite[Chapter 5]{MR3243734}.
Below, we add some details for \eqref{eq0514_01} and \eqref{eq0515_02}.

For \eqref{eq0514_01}, denote $D_{x_j} = D_j$.
For $\varphi \in C_0^\infty(\cQ)$, by Lemma \ref{Hilbert pp bound}
\[
\int_\cQ D_j H (u) \, \varphi \, dX = - \int_\cQ H(u) \, D_j \varphi \, dX = - \lim_{\varepsilon \to 0} \int_\cQ H^{(\varepsilon)} (u) \, D_j \varphi \, dX,
\]
where
\[
\int_\cQ H^{(\varepsilon)} (u) \, D_j \varphi \, dX = \int_\cQ \frac{1}{\pi} \int_{|s| \geq \varepsilon} \frac{u(t-s,x)}{s} \, ds \, D_j \varphi \, dX
\]
\[
= - \frac{1}{\pi} \int_{|s| \geq 
\varepsilon} \int_\cQ \frac{D_j u(t-s,x)}{s} \varphi(t,x) \, dX \, ds \to  - \int_\cQ H (D_ju) \, \varphi \, dX,
\]
as $\varepsilon \to 0$.
This proves \eqref{eq0514_01}.

To see \eqref{eq0515_02}, 
for $\varphi \in C_0^\infty(\cQ)$, we note that
\[
\int_\cQ D_j u \, D_t^{1/2}\varphi \, dX = \int_\cQ D_j u(t,x) \frac{1}{\sqrt{8\pi}}\int_\bR \frac{\varphi(t+\ell,x)-\varphi(t,x)}{|\ell|^{3/2}} \, d\ell \, dX
\]
\[
= \frac{1}{\sqrt{8\pi}} \int_\bR |\ell|^{-3/2} \int_Q \left(\varphi(t+\ell,x)-\varphi(t,x)\right) D_j u(t,x) \, dX \, d\ell
\]
\[
= - \frac{1}{\sqrt{8\pi}} \int_\bR |\ell|^{-3/2} \int_Q \left(D_j \varphi(t+\ell,x)- D_j \varphi(t,x)\right) u(t,x) \, dX \, d\ell
\]
\[
= - \int_\cQ u \, D_t^{1/2} D_j \varphi \, dX = - \int_\cQ D_t^{1/2}u \, D_j \varphi \, dX,
\]
which implies \eqref{eq0515_02}.
\end{proof}

\begin{lemma}
							\label{lem0515_1}
Let $1< p<\infty$ and $u \in L_p(\cQ)$ be such that $u_t \in L_p(\cQ)$.
Then, $D_t^{1/2}u$ and $D_t^{1/2}(D_t^{1/2}u)$ are in $L_p(\cQ)$.
\end{lemma}

\begin{proof}
If $\varphi \in C_0^\infty(\cQ)$, then
\[
\sqrt{8\pi}|D_t^{1/2}\varphi(t,x)| \leq \int_{|\ell|\leq 1} \frac{\left|\varphi(t+\ell,x)-\varphi(t,x)\right|}{|\ell|^{3/2}} \, d\ell
\]
\[
+ \int_{|\ell| > 1} \frac{\left|\varphi(t+\ell,x)-\varphi(t,x)\right|}{|\ell|^{3/2}} \, d\ell \leq \int_{|\ell| \leq 1}|\ell|^{-3/2} \int_0^\ell |\partial_t \varphi(t+s,x)| \, ds \, d\ell
\]
\[
+ \int_{|\ell| > 1} |\ell|^{-3/2} \left(|\varphi(t+\ell,x)| + |\varphi(t,x)|\right) \, d\ell.
\]
This shows that
\begin{equation}
							\label{eq0523_01}
\|D_t^{1/2}\varphi\|_{L_p(\cQ)} \leq N \|\partial_t \varphi\|_{L_p(\cQ)} + N \|\varphi\|_{L_p(\cQ)}.
\end{equation}
Now, if $u, u_t \in L_p(\cQ)$, one can find a sequence $u^k \in C_0^\infty(\cQ)$ such that
\[
u^k \to u, \quad \partial_t u^k \to \partial_t u \quad \text{in} \,\, L_p(\cQ),
\]
respectively, as $k \to \infty$.
This approximation with the estimate \eqref{eq0523_01} proves that $D_t^{1/2}u \in L_p(\cQ)$.

For any $\varphi \in C_0^\infty(\cQ)$, using Fourier transforms with \eqref{eq0519_01} and \eqref{eq0519_02}, we see that
\[
\int_\cQ D_t^{1/2} u^k \,  D_t^{1/2}\varphi \, dX
= \int_\cQ |\xi| \widetilde{u^k}(\xi,x) \overline{\widetilde \varphi(\xi,x)} \, d\xi \, dx
\]
\[
= \int_\cQ (-i)\operatorname{sgn}(\xi) \, i  \xi \, \widetilde{u^k}(\xi,x) \, \overline{\widetilde{\varphi}(\xi,x)} \, d\xi \, dx
= \int_\cQ H(\partial_t u^k) \, \varphi \, dX.
\]
Since $D_t^{1/2} u^k \to D_t^{1/2}u$ in $L_p(\cQ)$ from the  proof above, by letting $k \to \infty$ and considering the boundedness of the operator $H$, we have
\[
\int_\cQ D_t^{1/2}u \, D_t^{1/2}\varphi \, dX = \int_\cQ H(\partial_t u) \, \varphi \, dX,
\]
which shows that $D_t^{1/2}(D_t^{1/2}u) = H(\partial_t u) \in L_p(\cQ)$.
\end{proof}

\section{$L_2$-estimates}\label{sec L2 estimates}

In this section, we consider the operator
\[
\cP_\lambda:=u_t-D_i(a_{ij}D_j u)+\lambda u,
\]
where $a_{ij}$ satisfy the conditions in \eqref{eq0521_02} without any regularity assumptions on $a_{ij}$.
With the operator $\cP_\lambda$, in this section we present the $L_2$-solvability, which is essentially proved in \cite{MR0200593} and also seen in, for instance, \cite{MR3906170, MR4127944}.
For completeness, we provide a detailed proof.

\begin{theorem} \label{Interior L_2 solvability}
 Let $\Omega$ be a domain and $\cQ=\bR\times \Omega$. There exists $N=N(d,\delta)$ such that, for any $\lambda\ge 0$,
\begin{equation}
        \label{L_2 estiamte}
\|D_t^{1/2}u\|_2 + \|Du\|_2 + \sqrt{\lambda}\|u\|_2 \le N\left( \|h\|_2 + \|g_i\|_2 + \lambda^{-1/2}\|f\|_2\right),
\end{equation}
where $\|\cdot\|_2 = \|\cdot\|_{L_2(\cQ)}$,
provided that $u\in\mathring H^{1/2,1}_2(\cQ)$, $h, g_i, f \in L_2(\cQ)$, $i=1,\ldots,d$, with $f \equiv 0$ if $\lambda = 0$, and 
\begin{equation}\label{L_2 estimate eq}
\cP_\lambda u=D_t^{1/2}h+D_ig_i+f \quad \text{in} \,\, \cQ.
\end{equation}
Furthermore, for any $\lambda>0$, $h, g_i, f \in L_2(\cQ)$, $i=1,\ldots,d$, there exists a unique solution $u\in  \mathring{H}^{1/2,1}_2(\cQ)$ to the equation \eqref{L_2 estimate eq}. 
\end{theorem}
\begin{proof}
First, we consider the solvability.
Let $u,v\in {\mathring{H}_2^{1/2,1}(\cQ)}$. We define 
\[
\langle \cP_\lambda u,v\rangle
:= \int_\cQ {-H(D_t^{1/2}u)} \,D_t^{1/2}v+a_{ij}D_i uD_j v +\lambda uv \,dX
\]
and 
\[
\cB_\kappa[u,v]:=\langle \cP_\lambda u, (1 -\kappa H)v\rangle,
\]
where $\kappa$ is a positive constant and, in particular, $\langle\cP_\lambda u,\kappa H(v)\rangle$ is well-defined because $H(v) \in L_2(\cQ)$ and $D_t^{1/2}H(v) = H(D_t^{1/2}v) \in L_2(\cQ)$ by Lemmas \ref{Hilbert pp bound} and \ref{lem0514_1}.
Using denseness of $C_0^\infty(\cQ)$ on $\mathring H_2^{1/2,1}(\cQ)$, we have 
\[
\langle \cP_\lambda u,v\rangle =\int_\cQ hD_t^{1/2}v-g_iD_i v+f v\,dX=:\langle \cF,v\rangle
\]
if $u$ is a solution to the equation \eqref{L_2 estimate eq}.
Set
\[
U = (D_t^{1/2}u, Du, \sqrt{\lambda}u ) \quad V = ( D_t^{1/2}v, Dv, \sqrt{\lambda}v ).
\]
To apply the Lax-Milgram theorem, we claim that for sufficiently small $\kappa$, we have 
\begin{equation}
    \label{claim L2}
N\cB_\kappa[u,u]\ge
    \|U\|_{L_2(\cQ)}^2
\end{equation}
and 
\begin{equation}
    \label{claim L2 2}
\cB_\kappa[u,v]
\le
N\|U\|_{L_2(\cQ)}\|V\|_{L_2(\cQ)},
\end{equation}
where $N = N(d,\delta)$.
From the denseness of $C_0^\infty(\cQ)$ with respect to $\mathring H_2^{1/2,1}(\cQ)$, we may assume that $u\in C_0^\infty(\cQ)$ to prove the claim.
Upon noticing that
\[
\int_\cQ H(D_t^{1/2}u) \, D_t^{1/2}u \, dX = \int_\cQ u \, u_t \, dX = 0,
\]
where the first equality is due to Lemma \ref{lem0514_1},
we see that 
\[
\langle\cP_\lambda u,u\rangle \geq \delta \int_\cQ |Du|^2\,dX+\lambda\int_\cQ|u|^2\,dX.
\]
From Lemmas \ref{Hilbert pp bound} and \ref{lem0514_1} along with H\"{o}lder's inequality, we also have
\[
\langle\cP_\lambda u, -\kappa Hu\rangle
\ge 
\kappa\int_\cQ|D_t^{1/2}u|^2\,dX
-\delta^{-1}\kappa\int_\cQ|Du|^2\,dX
-\kappa \lambda \int_\cQ |u|^2.
\]
Then, by taking $\kappa = \delta^2/2$, we have \eqref{claim L2}.
For the estimate \eqref{claim L2 2}, we use H\"older's inequality and Lemmas \ref{Hilbert pp bound} and \ref{lem0514_1}.

Therefore, when $\lambda > 0$, by using the Lax-Milgram theorem, we have that there is a unique $u\in \mathring H_2^{1/2,1}(\cQ)$ such that for all $w\in \mathring H_2^{1/2,1}(\cQ)$,
\[
\cB_\kappa[u,w]=\langle \cF_\kappa,w\rangle,
\]
where 
\[
\langle \cF_\kappa,w\rangle:=\langle \cF, (1-\kappa H)w\rangle.
\]
For $v\in \mathring H_2^{1/2,1}(\cQ)$, let $w=(1+\kappa H)v/(1+\kappa^2)$, so that $(1-\kappa H)w=v$ because $H^2= -I$ (see Lemma \ref{lem0514_1}). Then, we have
\[
\langle \cP_\lambda u,v\rangle=\cB_\kappa [u,w]=\langle \cF_\kappa,w\rangle=\langle \cF,v\rangle.
\]
This implies that $u$ is a solution to the equation \eqref{L_2 estimate eq}.

To derive the estimate \eqref{L_2 estiamte} including the case $\lambda =0$, we use \eqref{claim L2} and H\"older's inequality to arrive at
\[
\|U\|_{L_2(\cQ)}^2\le N\cB_\kappa [u,u]=N\langle \cF_\kappa, u\rangle \le N\||h|+|g_i|+\lambda^{-1/2}|f|\|_{L_2(\cQ)}\|U\|_{L_2(\cQ)}.
\]
This completes the proof.
 \end{proof}

We now establish a local estimate.
Recall the notation $(u)_\cD$ in \eqref{eq0521_01}.

\begin{lemma}
\label{local L_2 estimate lemma}
     Let $R\in (0,\infty)$, and  $\cQ=(-\infty,\infty)\times B_R$. There exists $N=N(d,\delta)$ such that, for any $\lambda\ge 0$,
\[
\left(|D_t^{1/2}u|^2 + |Du|^2 + \lambda |u|^2\right)^{1/2}_{Q_R} \leq N \sum_{j=0}^\infty \left(|h|^2 + |g_i|^2 + \lambda^{-1}|f|^2\right)^{1/2}_{Q_{2^{j/2}R,R}},
\]
provided that $u\in\mathring H^{1/2,1}_2(\cQ)$, $h, g_i, f \in L_2(\cQ)$ with $f \equiv 0$ if $\lambda = 0$, and 
    \begin{equation*}
   \cP_\lambda u=D_t^{1/2}h+D_ig_i+f \quad \text{in} \,\, \cQ.
    \end{equation*}  
\end{lemma}
\begin{proof}
Thanks to scaling with the fact that
\[
D_t^{1/2}\left( u(R^2 t,Rx) \right) = R (D_t^{1/2} u)(R^2t,Rx),
\]
we may assume that $R=1$.
Let $\lambda_k=\max(2^{-k}-\lambda,0)$.
Using $\eta_k$ in \eqref{eq0513_02},
we observe that
\[
(u\eta_k)_t-D_i(a_{ij}D_j(u\eta_k))+(\lambda +\lambda_k)u\eta_k
    =D_t^{1/2}(h\eta_k)+D_i(g_i\eta_k)+f\eta_k   
\]
\begin{equation*}
 + h_k +u(\eta_k)_t+\lambda_ku\eta_k,
 \end{equation*}
in $\cQ$, where
\[
h_k(t,x) = -\frac{1}{\sqrt{8\pi}}\int_\bR h(t+\ell,x)\frac{\eta_k(t+\ell)-\eta_k(t)}{|\ell|^{3/2}} \, d \ell,
\]
which can also be written as
\begin{equation*}
h_k = \eta_k D_t^{1/2}h - D_t^{1/2}(h \eta_k),
\end{equation*}
provided that $D_t^{1/2}h \in L_2(\cQ)$.
By applying Theorem \ref{Interior L_2 solvability} to $u\eta_k$ and using \eqref{cut off bound} as well as the equality
\[
u(\eta_k)_t + \lambda_k u \eta_k = u(\eta_k)_t \eta_{k+1} + \lambda_k u \eta_k \eta_{k+1}
\]
in $\cQ$, we see that 
\[
\|D_t^{1/2}(u\eta_k)\|_{L_2(\cQ)}+\|Du\eta_k\|_{L_2(\cQ)}+\sqrt{\lambda}\|u\eta_k\|_{L_2(\cQ)}
\]
\[
\le \|D_t^{1/2}(u\eta_k)\|_{L_2(\cQ)}+\|Du\eta_k\|_{L_2(\cQ)}+\sqrt{\lambda+\lambda_k}\|u\eta_k\|_{L_2(\cQ)}
\]
\[
\le 
N\|h\eta_k\|_{L_2(\cQ)}
+N\|g\eta_k\|_{L_2(\cQ)}
+\frac{N}{\sqrt{\lambda+\lambda_k}}\|f\eta_k\|_{L_2(\cQ)}
\]
\[
+\frac{N}{\sqrt{\lambda+\lambda_k}}\|{h_k}\|_{L_2(\cQ)}+\frac{N(2^{-k}+\lambda_k)}{\sqrt{\lambda+\lambda_k}}\|u\eta_{k+1}\|_{L_2(\cQ)}
\]
\[
\le N\|h\eta_k\|_{L_2(\cQ)}
+N\|g\eta_k\|_{L_2(\cQ)}
+\frac{N}{\sqrt{\lambda}}\|f\eta_k\|_{L_2(\cQ)}
\]
\[
+N2^{k/2}\|h_k\|_{L_2(\cQ)}
+N2^{-k/2}\|u\eta_{k+1}\|_{L_2(\cQ)},
\]
where the last inequality is from the choice of $\lambda_k$.
Using the above estimates and the triangular inequality, we arrive at 
\[
\|\eta_k D_t^{1/2}u \|_{L_2(\cQ)}+\|Du\eta_k\|_{L_2(\cQ)}+\sqrt{\lambda}\|u\eta_k\|_{L_2(\cQ)}
\]
\[
\le 
N\|h\eta_k\|_{L_2(\cQ)}
+N\|g\eta_k\|_{L_2(\cQ)}
+\frac{N}{\sqrt{\lambda}}\|f\eta_k\|_{L_2(\cQ)}
\]
\[
+N2^{k/2}\|h_k\|_{L_2(\cQ)}
+N2^{-k/2}\|u\eta_{k+1}\|_{L_2(\cQ)}
\]
\begin{equation}
\label{local 13}
   +N\|D_t^{1/2}(u\eta_k)-\eta_k D_t^{1/2}u \|_{L_2(\cQ)}.   
\end{equation}

By Lemma \ref{tail estimate}, the following estimates hold:
\begin{equation}
    \label{local 21}
    \|D_t^{1/2}(u\eta_k)- {\eta_k}D_t^{1/2}u \|_{L_2(\cQ)}
    \le N2^{-k/2}
\sum_{j=1}^{\infty}2^{-j}\|u\|_{L_2((-2^{k+j},2^{k+j})\times B_1)}
\end{equation}
and
\begin{equation}
    \label{local 22}
    \|h_k\|_{L_2(\cQ)}
    \le N2^{-k/2}
\sum_{j=1}^{\infty}2^{-j}\|h\|_{L_2((-2^{k+j},2^{k+j})\times B_1)}.
\end{equation}
By combining \eqref{local 13}, \eqref{local 21}, and \eqref{local 22}, and setting $F = (h, g_i, \lambda^{-1/2}f)$, we establish that 
\[
\|D_t^{1/2}u \|_{L_2((-2^k,2^k)\times B_1)}+\|Du\|_{L_2((-2^k,2^k)\times B_1)}+\sqrt{\lambda}\|u\|_{L_2((-2^k,2^k)\times B_1)}
\]
\[
\le 
\|\eta_k D_t^{1/2}u \|_{L_2(\cQ)}+\|Du\eta_k\|_{L_2(\cQ)}+\sqrt{\lambda}\|u\eta_k\|_{L_2(\cQ)}
\]
\[
\le N2^{-k/2}
\sum_{j=1}^{\infty}2^{-j}\|u\|_{L_2((-2^{k+j},2^{k+j})\times B_1)}+ N
\sum_{j=1}^{\infty}2^{-j}\|F\|_{L_2((-2^{k+j},2^{k+j})\times B_1)}
\]\[
\le N2^{-k/2}
\sum_{j=1}^{\infty}2^{-j}\|Du\|_{L_2((-2^{k+j},2^{k+j})\times B_1)}+ N
\sum_{j=1}^{\infty}2^{-j}\|F\|_{L_2((-2^{k+j},2^{k+j})\times B_1)},
\]
where we use the zero boundary condition and the Sobolev embedding theorem for $u$ to derive the last inequality.
By multiplying the above inequality by $2^{-k/2}$, we have
\begin{equation}
\label{local 23}
A_k
\le
N_1\sum_{j=k+1}^{\infty}2^{-j/2}A_j
+N_1 2^{k/2}\sum_{j=k+1}^{\infty}2^{-j/2}F_j,
\end{equation}
where  $A_j:=\left(|D_t^{1/2}u|^2 + |Du|^2 + \lambda |u|^2\right)_{(-2^j,2^j)\times B_1}^{1/2}$, 
$F_j:=\left(|F|^2\right)_{(-2^j,2^j)\times B_1}^{1/2}$, $j=0,1,2,\dots$, and $N_1=N_1(d,\delta)$. We choose a positive integer $k_0$ such that
\[
N_1\sum_{k=k_0}^\infty 2^{-k/2}\le 1/2.
\]
By multiplying the above inequality by $2^{-k/2}$, and summing in $k=k_0,k_0+1,\dots$, we obtain that 
\[
\sum_{k=k_0}^{\infty}2^{-k/2}A_k
\le 
N_1\sum_{k=k_0}^{\infty}2^{-k/2}\sum_{j=k+1}^{\infty}2^{-j/2}A_j
+N_1\sum_{k=k_0}^{\infty}\sum_{j=k+1}^{\infty}2^{-j/2}F_j
\]
\[
=
N_1\sum_{j=k_0+1}^{\infty}\left(2^{-j/2}A_j\sum_{k=k_0}^{j-1}2^{-k/2}\right)
+N_1\sum_{j=k_0+1}^{\infty}\left(2^{-j/2}{F_j}\sum_{k=k_0}^{j-1}1\right)
\]
\[
\le \frac{1}{2}\sum_{j=k_0+1}^{\infty}2^{-j/2}A_j
+N_1\sum_{j=k_0+1}^{\infty}(j-k_0)2^{-j/2}F_j.
\]
\[
\le
\frac{1}{2}\sum_{j=k_0+1}^{\infty}2^{-j/2}A_j
+N\sum_{j=k_0+1}^{\infty}2^{-j/4}F_j.
\]
Thus, it follows that 
\[
\sum_{k=k_0}^{\infty}2^{-k/2}A_k
\le 
N\sum_{j=0}^{\infty}2^{-j/4}F_j.
\]
Finally, for the values $k=k_0-1, k_0-2, \dots, 0$, by utilizing the inequality \eqref{local 23} and induction, we finish the proof.
\end{proof}

\section{Mean oscillation estimates: time measurable $a_{ij}$}							\label{sec05}

In this section, we obtain mean oscillation estimates of solutions to equations as in \eqref{main eq} when the operator is
\[
\cP_\lambda u=u_t- D_i(a_{ij}D_j u)+\lambda u,
\]
where $a_{ij}=a_{ij}(t)$.
We emphasize that the coefficients $a_{ij}(t)$ are functions of only $t \in \bR$, satisfying the ellipticity and boundedness conditions in \eqref{eq0521_02}, but without any regularity conditions.

We begin by establishing Lipschitz estimates for $Du$ and $\sqrt{\lambda}u$ when $u$ is a solution to a homogeneous equation. Subsequently, we derive mean oscillation estimates for $Du$ and $\sqrt{\lambda}u$.
After that, for a homogeneous solution $u$, we obtain Lipschitz estimates for $D_t^{1/2}u$, $Du$, and $\sqrt{\lambda}u$ under the assumption that  $a_{ij}=\delta_{ij}$ and achieve mean oscillation estimates for these terms.

For a function $u$ defined on $\cD\subset \bR^{d+1}$, we denote, for $\mu\in(0,1]$, 
\[
[u]_{C^\mu(\cD)}=\underset{(t,x)\neq (s,y)}{\sup_{(t,x),(s,y)\in \cD}}\frac{|u(t,x)-u(s,y)|}{{|t-s|^{\mu/2}}+|x-y|^\mu}.
\]
We borrow the following $L_\infty$-estimate of the spatial derivatives of $u$ and $u_t$ from \cite{MR2771670}.

\begin{lemma}
\label{ARMA original Lemma 2}
Let $u\in W_2^{1,\infty}(Q_4)$ satisfy 
    \[
    \cP_0 u=0 \quad \text{in} \,\, Q_4.
    \]
    Then, for any multi-index $\alpha$, we have 
    \[
    \|D^\alpha u\|_{L_\infty(Q_1)}+\|D^\alpha u_t\|_{L_\infty(Q_1)}\le N\|u\|_{L_2(Q_4)},
    \]
    where $N=N(d,\delta,\alpha)$.
\end{lemma}

\begin{proof}
The proof is the same as in the proof of Lemma 2 in \cite{MR2771670}. Indeed, the authors have shown that for any multi-index $\alpha$,
\[
\|D^\alpha u\|_{L_\infty(Q_1)}\le N(d,\alpha) \|u\|_{L_2(Q_2)}.
\]
The estimate for the supremum of $D^\alpha u_t$  follows from the equation $u_t=D_i(a_{ij}(t)D_j u)=a_{ij}(t)D_{ij}u$.  
\end{proof}

\begin{lemma}
    \label{holder estimate of u}
Let  $u\in W_2^{1,\infty}(Q_4)$ satisfy 
    \[
   \cP_0 u=0 \quad \text{in} \,\, Q_4.
    \]
    Then, we have 
    \[
    [u]_{C^1(Q_1)}\le N\|u\|_{L_2(Q_4)},
    \]
    where $N=N(d,\delta)$.
 \end{lemma}

 \begin{proof}
 Let $(t,x),(s,y)\in Q_1$. 
 By using the triangle inequality, the fundamental theorem of calculus, and Lemma \ref{ARMA original Lemma 2}, we have
 \[
 |u(t,x)-u(s,y)|\le |u(t,x)-u(t,y)|+|u(t,y)-u(s,y)|
 \]
 \[
 \le N \|Du\|_{L_{\infty}(Q_1)}|x-y|+N\|u_t\|_{L_\infty(Q_1)}|t-s|
 \]
 \[
  \le N \|Du\|_{L_{\infty}(Q_1)}|x-y|+N\|u_t\|_{L_\infty(Q_1)}\sqrt{|t-s|}
  \]
 \[
 \le N(|x-y|+\sqrt{|t-s|})\|u\|_{L_2(Q_4)}.
 \]
Dividing the first and last terms in the above inequalities by $(|x-y| + \sqrt{|t-s|})$ yields the desired inequality. 
 \end{proof}

\begin{lemma}
\label{interior holder lemma}
     There exists $N=N(d,\delta)$ such that, for any $\lambda\ge 0$,
\[
\left[Du\right]_{C^1(Q_1)} + \sqrt{\lambda}[u]_{C^1(Q_1)} \leq N \left(\|Du\|_{L_2(Q_4)} + \sqrt{\lambda}\|u\|_{L_2(Q_4)}\right),
\]
provided that $u\in H^{1/2,1}_2( \bR\times B_4)$ and 
    \begin{equation*}
    \cP_\lambda u=0 \quad \text{in} \,\, \bR\times B_4.
    \end{equation*}
\end{lemma}

\begin{proof}
     We only consider the case $\lambda=0$.
For the case $\lambda>0$, we follow S. Agmon's approach as in the proof of \cite[Lemma 3]{MR2771670}.
     We may assume that $u\in W_2^{1,\infty}(\bR\times B_4)$. 
     By observing that $Du$ is also a solution to the equation $u_t-D_i(a_{ij}(t)D_j u)=0$ and using Lemma \ref{holder estimate of u}, we finish the proof.
\end{proof}

We are now prepared to derive mean oscillation estimates.

\begin{lemma}
    \label{mean oscillation cal U}
    Let $r\in(0,\infty)$, $\kappa \in [4,\infty)$, $\lambda\ge0$, and $X_0=(t_0,x_0)\in \bR^{d+1}$. Assume that $u\in H^{1/2,1}_2(\bR^{d+1})$ satisfies 
    \[
    \cP_\lambda u=D_t^{1/2}h+D_ig_i+f \quad \text{in} \,\, \bR\times B_{\kappa r}(x_0), 
    \]
    where $h, g_i, f \in L_2(\bR^{d+1})$  with $f \equiv 0$ if $\lambda = 0$.
Then, we have
\[
\left(|Du - (Du)_{Q_r(X_0)}| \right)_{Q_r(X_0)} + \left(|\sqrt{\lambda}u - (\sqrt{\lambda}u)_{Q_r(X_0)}| \right)_{Q_r(X_0)}
\]
\[
\le N\kappa^{1+d/2}\sum_{j=0}^\infty2^{-j/4}(|F|^2)^{1/2}_{Q_{2^{j/2}\kappa r,\kappa r}(X_0)}+N\kappa^{-1} (|\cU|^2)^{1/2}_{Q_{\kappa r}(X_0)},
\]
where $N=N(d,\delta)$, 
\begin{equation}
							\label{eq1010_01}
F = (h, g_i, \lambda^{-1/2}f), \quad \cU = (Du, \sqrt{\lambda}u).
\end{equation}
\end{lemma}

\begin{proof}
    By dilation and translation, we may assume that $r=4/\kappa$ and $X_0=0$.  By taking $\lambda \searrow 0$, we only consider the case $\lambda>0$.  
    By utilizing Theorem \ref{Interior L_2 solvability}, we find a unique solution $w\in \mathring H_2^{1/2,1}(\bR\times B_{\kappa r})$ satisfying 
    \[
     \cP_\lambda w=D_t^{1/2} h+D_ig_i+f \quad \text{in} \,\, \bR\times B_{\kappa r}.
    \]
    Set $v=u-w$, which belongs to $H_2^{1/2,1}(\bR \times B_{\kappa r})$ and satisfies
\begin{equation}
							\label{eq0521_03}
\cP_\lambda v = 0 \quad \text{in} \,\, \bR \times B_{\kappa r}.
\end{equation}
 Denote
    \[
    \cW=(D_1w,\dots D_d w,\sqrt{\lambda}w), \quad \cV=(D_1v,\dots D_d v,\sqrt{\lambda}v). 
    \]
    From H\"older's inequality and Lemma \ref{local L_2 estimate lemma}, it follows that 
    \[
    (|\cW-(\cW)_{Q_r}|)_{Q_r}\le N(|\cW|^2)^{1/2}_{Q_r}\le N\kappa^{1+d/2}(|\cW|^2)^{1/2}_{Q_{\kappa r}}\]
       \begin{equation}
        \label{estimate w whole}
    \le N\kappa^{1+d/2}\sum_{j=0}^\infty 2^{-j/4}(|F|^2)^{1/2}_{Q_{2^{j/2}\kappa r,\kappa r}}.
         \end{equation}
    By applying Lemma \ref{interior holder lemma} to \eqref{eq0521_03} with $r=4/\kappa$, we have
    \[
    (|\cV-(\cV)_{Q_r}|)_{Q_r}\le N r [\cV]_{C^1(Q_1)}\le N\kappa^{-1} (|\cV|^2)^{1/2}_{Q_4}=N\kappa^{-1} (|\cV|^2)^{1/2}_{Q_{\kappa r}}.
    \]
    By using $v=u-w$ for the last term of the above estimate, we see that 
    \[
    N\kappa^{-1} (|\cV|^2)^{1/2}_{Q_{\kappa r}}\le N\kappa^{-1} (|\cW|^2)^{1/2}_{Q_{\kappa r}}+N\kappa^{-1} (|\cU|^2)^{1/2}_{Q_{\kappa r}}
    \]
    \[
    \le  N\kappa^{-1}\sum_{j=0}^\infty2^{-j/4}(|F|^2)^{1/2}_{Q_{2^{j/2}\kappa r,\kappa r}}+N\kappa^{-1} (|\cU|^2)^{1/2}_{Q_{\kappa r}}.
    \]
    Hence,
    \begin{equation}
        \label{estiamte v whole}
        (|\cV-(\cV)_{Q_r}|)_{Q_r}\le  N\kappa^{-1}\sum_{j=0}^\infty2^{-j/4}(|F|^2)^{1/2}_{Q_{2^{j/2}\kappa r,\kappa r}}+N\kappa^{-1} (|\cU|^2)^{1/2}_{Q_{\kappa r}}.
     \end{equation}
Combining the estimates \eqref{estimate w whole} and \eqref{estiamte v whole}, and using the triangle inequality, we reach that 
\[
(|\cU-(\cU)_{Q_r}|)_{Q_r}\le (|\cW-(\cW)_{Q_r}|)_{Q_r}+(|\cV-(\cV)_{Q_r}|)_{Q_r}
\]
\[
\le N\kappa^{1+d/2}\sum_{j=0}^\infty 2^{-j/4}(|F|^2)^{1/2}_{Q_{2^{j/2}\kappa r,\kappa r}}
  +N\kappa^{-1} (|\cU|^2)^{1/2}_{Q_{\kappa r}}.
    \]
The lemma is proved.
\end{proof}

Here, we consider the case 
\[
\cP_\lambda u= u_t-\Delta u+\lambda u.
\]
In this case, we estimate the mean oscillations of $D_t^{1/2}u$, $Du$, and $\sqrt{\lambda}u$.
 In other words, when the coefficients of the operator are simple, we estimate the mean oscillations including $D_t^{1/2} u$.
We start by establishing Lipshitz estimates including $D_t^{1/2} u$ term.
 
\begin{lemma}
    \label{interior holder simple}
     There exists $N=N(d)$ such that, for any $\lambda\ge 0$,
\[
[D_t^{1/2}u]_{C^1(Q_1)} + \left[Du\right]_{C^1(Q_1)} + \sqrt{\lambda}[u]_{C^1(Q_1)} \leq N \||D_t^{1/2}u| + |Du|+\sqrt{\lambda}|u|\|_{L_2(Q_4)}
\]
provided that $u\in H^{1/2,1}_2( \bR\times B_4)$, and 
    \begin{equation*}
    u_t-\Delta u+\lambda u=0 \quad \text{in} \,\, \bR\times B_4.
    \end{equation*}
\end{lemma}

\begin{proof}
By the well-known $L_p$ theory for the heat equation, it is clear that $u \in C^\infty(\cQ)$ and its partial derivatives belong to $L_2(\cQ)$, where $\cQ = \bR \times B_r$, $r \in (0,4)$.
Thus, Lemma \ref{lem0515_1} implies $D_t^{1/2}u \in H_2^{1/2,1}(\cQ)$ because $u, u_t \in L_2(\cQ)$ and $Du, \partial_t Du \in L_2(\cQ)$ along with the identity \eqref{eq0515_02}.
By (formally) taking the half-time derivative of both sides of the heat equation, one can verify that $D_t^{1/2}u$ is also a solution to
\[
u_t - \Delta u + \lambda u = 0
\]
in $\cQ = \bR \times B_r$, $r \in (0,4)$.
Then, following the proof of Lemma \ref{interior holder lemma}, we conclude the proof.
\end{proof}

We are now set to obtain mean oscillation estimates for $D_t^{1/2}u$, $Du$, and $\sqrt{\lambda}u$, assuming that $D_i(a_{ij}D_j u)=\Delta u$. The proof is identical to that of Lemma \ref{mean oscillation cal U}, using Lemma
\ref{interior holder simple} instead of Lemma \ref{interior holder lemma}. Thus, we omit the proof.

\begin{lemma}
    \label{mean oscillation U simple}
    Let $r\in(0,\infty)$, $\kappa \in [4,\infty)$, $\lambda\ge0$, and $X_0=(t_0,x_0)\in \bR^{d+1}$. Assume that $u\in H^{1/2,1}_2(\bR^{d+1})$ satisfies 
    \[
     u_t-\Delta u+\lambda u=D_t^{1/2}h+D_ig_i+f \quad \text{in} \,\, \bR\times B_{\kappa r}(x_0), 
    \]
    where $h, g_i, f \in L_2(\bR^{d+1})$ with $f \equiv 0$ if $\lambda = 0$.
    Then, we have 
\[
\left(|D_t^{1/2}u - (D_t^{1/2}u)_{Q_r(X_0)}| \right)_{Q_r(X_0)} + 
\left(|Du - (Du)_{Q_r(X_0)}| \right)_{Q_r(X_0)}
\]
\[
+ \left(|\sqrt{\lambda}u - (\sqrt{\lambda}u)_{Q_r(X_0)}| \right)_{Q_r(X_0)}
\]
\[
\le N\kappa^{1+d/2}\sum_{j=0}^\infty2^{-j/4}(|F|^2)^{1/2}_{Q_{2^{j/2}\kappa r,\kappa r}(X_0)}+N\kappa^{-1} (|U|^2)^{1/2}_{Q_{\kappa r}(X_0)},
\]
where $N=N(d)$, $F$ is as in \eqref{eq1010_01}, and
\begin{equation}
							\label{eq1010_02}
U = ( D_t^{1/2}u, Du, \sqrt{\lambda}u).
\end{equation}
\end{lemma}

\section{Proof of Theorem \ref{main whole space}}
							\label{sec06}

In this section, we prove Theorem \ref{main whole space}.
We begin by establishing mean oscillation estimates for the operator
\[\cP_\lambda u=u_t-D_i(a_{ij}D_j u)+\lambda u,\]
where the coefficients satisfy Assumption \ref{assum coeffi}.
We then employ the method of freezing coefficients.
Note that the constant $R_0$, which appears below, is the one from Assumption \ref{assum coeffi}.

\begin{lemma}\label{mean oscillation general coeffi}
Let $\nu\in (2,\infty)$, $\nu'=2\nu/(\nu-2)$, $r\in(0,R_0/\kappa]$, $\kappa \in [4,\infty)$, $\lambda\ge0$, and $X_0=(t_0,x_0)\in \bR^{d+1}$.
Assume that $u \in H^{1/2,1}_\nu(\bR^{d+1}) \cap H^{1/2,1}_2(\bR^{d+1})$ satisfies 
    \[
    \cP_\lambda u=D_t^{1/2}h+D_ig_i+f \quad \text{in} \,\, \bR^{d+1}, 
    \]
    where $h, g_i, f \in L_{\nu}(\bR^{d+1}) \cap L_2(\bR^{d+1})$ with $f \equiv 0$ if $\lambda = 0$.
    Then,
    under Assumption \ref{assum coeffi} ($\gamma$), we have
\[
\left(|Du - (Du)_{Q_r(X_0)}| \right)_{Q_r(X_0)} + \left(|\sqrt{\lambda}u - (\sqrt{\lambda}u)_{Q_r(X_0)}| \right)_{Q_r(X_0)}
\]
\[
\le N\kappa^{1+d/2}\sum_{j=0}^\infty2^{-j/4}(|F|^2)^{1/2}_{Q_{2^{j/2}\kappa r,\kappa r}(X_0)}
\]
\[
+N\kappa^{-1} (|\cU|^2)^{1/2}_{Q_{\kappa r}(X_0)}+N\gamma^{1/\nu'}\kappa^{1+d/2}\sum_{j=0}^\infty2^{-j/4}(|Du|^\nu)_{Q_{2^{j/2}R,R}(X_0)}^{1/\nu},
\]
where $\cU$ and $F$ are as in \eqref{eq1010_01}, and $N=N(d,\delta,\nu)$.
\end{lemma}

\begin{proof}
    Denote $R=\kappa r(\le R_0)$. We may assume that $X_0=0$.
    Let 
    \[
    \overline{\cP_\lambda}u=u_t-D_i(\bar{a}_{ij}(t)D_j u)+\lambda u,
    \]
    where 
    \[
   \bar{a}_{ij}(t):=\dashint_{B_{R}}a_{ij}(t,y)\,dy.
    \]
    Then, we observe that 
    \begin{equation}
        \label{bar eq}
    \overline{\cP_\lambda}u=D_t^{1/2}h+D_i\tilde{g}_i+h \quad \text{in} \,\, \bR^{d+1}, 
        \end{equation}
    where 
    \[
    \tilde{g}_i=g_i+(a_{ij}-\bar{a}_{ij})D_j u.
    \]
By applying Lemma \ref{mean oscillation cal U} to $u$ with the equation \eqref{bar eq}, we have 
\[
\left(|Du - (Du)_{Q_r}| \right)_{Q_r} + \left(|\sqrt{\lambda}u - (\sqrt{\lambda}u)_{Q_r}| \right)_{Q_r}
\]
\[
\le N\kappa^{1+d/2}\sum_{j=0}^\infty2^{-j/4}(|F|^2)^{1/2}_{Q_{2^{j/2}R,R}}+N\kappa^{-1} (|\cU|^2)^{1/2}_{Q_R}
    \]
    \[
    +N\kappa^{1+d/2}\sum_{l=0}^\infty2^{-l/4}(|a_{ij}-\bar{a}_{ij}|^2|Du|^2)^{1/2}_{Q_{2^{l/2}R,R}}.
    \]
Using H\"older's inequality, $|a_{ij}|\le \delta^{-1}$, and Assumption \ref{assum coeffi} ($\gamma$), we see that 
\[
(|a_{ij}-\bar{a}_{ij}|^2|Du|^2)_{Q_{2^{l/2}R,R}}^{1/2}
\le (|a_{ij}-\bar{a}_{ij}|^{\nu'})_{Q_{2^{l/2}R,R}}^{1/\nu'}
(|Du|^\nu)_{Q_{2^{l/2}R,R}}^{1/\nu}
\]
\[
\le  N(|a_{ij}-\bar{a}_{ij}|)_{Q_{2^{l/2}R,R}}^{1/\nu'}
(|Du|^\nu)_{Q_{2^{l/2}R,R}}^{1/\nu} \leq N \gamma^{1/\nu'}
(|Du|^\nu)_{Q_{2^{l/2}R,R}}^{1/\nu},
\]
where for the last inequality, see Remark \ref{rem0522_01}.
From this, we obtain the inequality in the lemma with $X_0 = 0$. We finish the proof.
\end{proof}

For a function $f$ defined on $\bR^{d+1}$, we denote its (parabolic) maximal and strong maximal functions, respectively, by 
\[
\cM f(t,x)=\sup_{Q_r(\tau,z)\ni (t,x)} \dashint_{Q_r(\tau,z)}|f(s,y)|\,dy\,ds,
\]
and 
\[
\cS\cM f(t,x)=\sup_{Q_{r,s}(\tau,z)\ni (t,x)}\dashint_{Q_{r,s}(\tau,z)}|f(s,y)|\,dy\,ds.
\]

Next, we define the following filtration of partitions in $\bR^{d+1}$:
\[
\mathbb{C}_n:=\{Q^n=Q^n_{(i_0,i_1,\dots,i_d)}:(i_0,i_1,\dots,i_d)\in \mathbb{Z}^{d+1}\},
\]
where $n\in \mathbb{Z}$, and 
\[
Q^n_{(i_0,i_1,\dots,i_d)}
\]
\[
=[2i_02^{-2n},2(i_0+1)2^{-2n})\times[i_12^{-n},(i_1+1)2^{-n})\times \cdots \times [i_d2^{-n},(i_d+1)2^{-n}).
\]

\begin{remark}
    \label{comparision cube}    
We note that for $X\in \bR^{d+1}$ and $X\in Q^n\in \mathbb{C}_n$, one can find $X_0\in \bR^{d+1}$ and the smallest $r>0$ (compatible with $2^{-n}$) such that $Q^n\subset Q_r(X_0)$ and 
\begin{equation*}
\dashint_{Q^n}|f-(f)_{Q^n}|\,dY\le N \dashint_{Q_r(X_0)}|f-(f)_{Q_r(X_0)}|\,dY,    
\end{equation*}
where $N=N(d)$.
\end{remark}

We denote the dyadic sharp function of $f$ by 
\[
f^{\#}_{dy}(t,x)=\sup_{n<\infty}\dashint_{Q^n\ni (t,x)}|g(s,y)-g_{|n}(t,x)|\,dy\,ds,
\]
where 
\[
g_{|n}(t,x)=\dashint_{Q^n}g(s,y)\,dy\,ds, \quad (t,x)\in Q^n.
\]

\begin{lemma}
    \label{small support}
    Let $\nu\in (2,\infty)$, $\nu'=2\nu/(\nu-2)$,  $\kappa \in [4,\infty)$, $\rho\in (0,1)$, $\lambda\ge0$, and $X\in \bR^{d+1}$. Assume that $u \in H^{1/2,1}_\nu(\bR^{d+1}) \cap H^{1/2,1}_2(\bR^{d+1})$ vanishes outside $\bR\times B_{\rho R_0}$ and satisfies 
    \[
    \cP_\lambda u=D_t^{1/2}h+D_ig_i+f \quad \text{in} \,\, \bR^{d+1}, 
    \]
    where $h, g_i, f \in L_{\nu}(\bR^{d+1}) \cap L_2(\bR^{d+1})$ with $f\equiv 0$ if $\lambda = 0$.
    Then,
    under Assumption \ref{assum coeffi} ($\gamma$), we have
\[
(Du)^{\#}_{dy}(X) + (\sqrt{\lambda} u)^{\#}_{dy}(X) \le N\kappa^{1+d/2}(\cS\cM (|F|^2)(X))^{1/2}\]
\[
+N(\kappa^{-1}+\rho^{d/2}\kappa^{d/2})(\cM(|\cU|^2)(X))^{1/2} +N\gamma^{1/\nu'}\kappa^{1+d/2}(\cS\cM(|\cU|^\nu)(X))^{1/\nu},
\]
where $\cU$ and $F$ are as in \eqref{eq1010_01}, and $N=N(d,\delta,\nu)$.
\end{lemma}

\begin{proof}
Let $X\in Q^n$.
Set $\Phi$ to be either $Du$ or $\sqrt{\lambda}u$.
By Remark \ref{comparision cube}, we choose $X_0\in \bR^{d+1}$ and the smallest $r>0$ such that 
     $Q^n\subset Q_r(X_0)$ and 
\begin{equation}
\label{small support eq1}
\dashint_{Q^n}|\Phi-(\Phi)_{Q^n}|\,dY
\le
N \dashint_{Q_r(X_0)}|\Phi-(\Phi)_{Q_r(X_0)}|\,dY.    
\end{equation}

First, we consider the case $r>R_0/\kappa$. 
By direct calculations, we see that 
\[
     \dashint_{Q_r(X_0)}|\Phi-( \Phi)_{Q_r(X_0)}|\,dY
     \le
     2\dashint_{Q_r(X_0)}|\cU|\,dY
     =
     2\dashint_{Q_r(X_0)}I_{B_{\rho R_0}}|\cU|\,dY
     \]
     \[
     \le 2\left(\dashint_{Q_r(X_0)}I_{B_{\rho R_0}}\,dY\right)^{1/2}\left(\dashint_{Q_r(X_0)}|\cU|^2\,dY\right)^{1/2}      
     \]
\begin{equation}
    \label{small support eq2}
\le N\rho^{d/2}(R_0/r)^{d/2}(\cM(|\cU|^2)(X))^{1/2}
\le N\rho^{d/2}\kappa^{d/2}(\cM(|\cU|^2)(X))^{1/2}.
     \end{equation}

Next, we consider the case $r\in(0, R_0/\kappa]$.
By Lemma \ref{mean oscillation general coeffi}, we have
  \[
    \dashint_{Q_r(X_0)}|\Phi-(\Phi)_{Q_r(X_0)}|\,dY
    \le N\kappa^{1+d/2}\sum_{j=0}^\infty2^{-j/4}(|F|^2)^{1/2}_{Q_{2^{j/2}\kappa r,\kappa r}(X_0)}
       \]
    \[ +N\kappa^{-1} (|\cU|^2)^{1/2}_{Q_{\kappa r}(X_0)}
      +N\gamma^{1/\nu'}\kappa^{1+d/2}\sum_{j=0}^\infty2^{-j/4}(|Du|^\nu)_{Q_{2^{j/2}R,R}(X_0)}^{1/\nu}
    \]
   \[\le
    N\kappa^{1+d/2}(\cS\cM (|F|^2)(X))^{1/2}+N\kappa^{-1} (\cM(|\cU|^2)(X))^{1/2}
    \]
    \begin{equation}
        \label{small support eq3}
          +N\gamma^{1/\nu'}\kappa^{1+d/2}(\cS\cM(|\cU|^\nu)(X))^{1/\nu}.
\end{equation}
Finally, by combining the estimates \eqref{small support eq1}, \eqref{small support eq2}, and \eqref{small support eq3}, and taking the supremum with respect to $Q^n\ni X$, we finish the proof.    
\end{proof}

\begin{lemma}
    \label{Lp small support}
    Let $p\in (2,\infty)$, $\lambda \geq 0$, $h, g_i, f \in L_p(\bR^{d+1}) \cap L_2(\bR^{d+1})$ with $f \equiv 0$ if $\lambda = 0$.
There exist positive constants $\gamma$, $\rho$, and $N$ depending only on $d$, $\delta$, and $p$ such that, under Assumption \ref{assum coeffi} ($\gamma$), for any $u \in H^{1/2,1}_p(\bR^{d+1}) \cap H^{1/2,1}_2(\bR^{d+1})$ vanishing outside $\bR\times B_{\rho R_0}$ and satisfying 
    \[
    \cP_\lambda u=D_t^{1/2}h+D_ig_i+f \quad \text{in} \,\, \bR^{d+1}, 
    \]
we have
\[
\|Du\|_p + \sqrt{\lambda}\|u\|_p \leq N \left( \|h\|_p + \|g_i\|_p + \lambda^{-1/2}\|f\|_p\right),
\]
where $\|\cdot\|_p = \|\cdot\|_{L_p(\bR^{d+1})}$.
\end{lemma}

    \begin{proof}
        Let $0<\gamma,\rho<1$ and $\kappa\ge4$ be constants to be specified later. Take an exponent $\nu$ such that $p>\nu>2$.
Note that by interpolation, $h,g_i,f \in L_\nu(\bR^{d+1})$ and $u \in H_\nu^{1/2,1}(\bR^{d+1})$.
Set $\Phi$ to be either $Du$ or $\sqrt{\lambda}u$, and let $F$ and $\cU$ be as in \eqref{eq1010_01}.
By using Lemma \ref{small support}, the Hardy-Littlewood maximal function theorem, and the Fefferman-Stein sharp function theorem (see, for example, \cite[Chapter 4]{MR2435520}), we obtain that 
        \[
        \|\Phi\|_{L_p(\bR^{d+1})}
        \le N\|\Phi^{\#}_{dy}\|_{L_p(\bR^{d+1})}
        \le N\kappa^{1+d/2}\|\cS\cM (|F|^2)^{1/2}\|_{L_p(\bR^{d+1})}
        \]
        \[+N(\kappa^{-1}+\rho^{d/2}\kappa^{d/2})\| \cM(|\cU|^2)^{1/2}\|_{L_p(\bR^{d+1})}        +N\gamma^{1/\nu'}\kappa^{1+d/2}\|\cS\cM(|\cU|^\nu)^{1/\nu}\|_{L_p(\bR^{d+1})}
        \]
    \[\le 
    N_0\kappa^{1+d/2}\|F\|_{L_p(\bR^{d+1})}        +N_0(\kappa^{-1}+\rho^{d/2}\kappa^{d/2}+\gamma^{1/\nu'}\kappa^{1+d/2})\| \cU\|_{L_p(\bR^{d+1})},
        \]
        where $N_0=N_0(d,\delta,p)$.
        Take $\kappa=4N_0$ and then choose sufficiently small $\gamma$ and $\rho$, so that 
        \[
        N_0\gamma^{1/\nu'}\kappa^{1+d/2}\le 1/4,\quad N_0\rho^{d/2}\kappa^{d/2}\le 1/4.
        \]
        This concludes the proof.
    \end{proof}

By employing the standard partition of unity argument with respect to the spatial variables, we establish the following theorem.

     \begin{theorem}
          \label{Lp mathcal U}
Let $p\in (2,\infty)$, $\lambda \geq 0$, $h, g_i, f \in L_p(\bR^{d+1}) \cap L_2(\bR^{d+1})$ with $f \equiv 0$ if $\lambda = 0$.
There exist positive constants $\gamma$  and $N$ depending only on $d$, $\delta$, and $p$ such that, under Assumption \ref{assum coeffi} ($\gamma$), for any $u \in H^{1/2,1}_p(\bR^{d+1}) \cap H^{1/2,1}_2(\bR^{d+1})$ satisfying 
    \[
     \cP_\lambda u=D_t^{1/2}h+D_ig_i+f \quad \text{in} \,\, \bR^{d+1}, 
    \]
    we have
\begin{equation}
        \label{est Lp mathcal U}
\|Du\|_p + \sqrt{\lambda}\|u\|_p \leq N \left( \|h\|_p + \|g_i\|_p + \lambda^{-1/2}\|f\|_p\right),
\end{equation}
where $\|\cdot\|_p = \|\cdot\|_{L_p(\bR^{d+1})}$
provided that $\lambda\ge \lambda_0$, where $\lambda_0=\lambda_0(d,\delta,p,R_0)\geq0$.
     \end{theorem}

    \begin{proof}
        For clarity, we use $N_0$ to denote the constants depending on $d$, $\delta$, $p$, and $R_0$.
        Take  $\zeta\in C_0^\infty(\bR^d)$ satisfying 
        \begin{equation}
            \label{eq4 Lp mathcal U}
        \|\zeta\|_{L_p(\bR^d)}=1,
        \end{equation}
 with support in the ball $B_{\rho R_0}$, where $\rho=\rho(d,\delta,p) \in (0,1)$ is the constant taken from Lemma \ref{Lp small support}.
        We observe that 
         \[
         |u(t,x)|^p=\int_{\bR^d}|u(t,x)\zeta(x-y)|^p\,dy
         \]
         and 
        \[
        |Du(t,x)|^p=\int_{\bR^d}|Du(t,x)\zeta(x-y)|^p\,dy
        \]
        \[
        \le \int_{\bR^d}|D(u(t,x)\zeta(x-y))|^p\,dy
        +N\int_{\bR^d}|u(t,x)D\zeta(x-y)|^p\,dy
        \]
        \[
        \le \int_{\bR^d}|D(u(t,x)\zeta(x-y))|^p\,dy
        +N_0|u(t,x)|^p,
         \]
        where we use 
        \begin{equation}
            \label{eq2 Lp mathcal U}
            \int_{\bR^d}|D\zeta(x)|^p\,dx \le N_0 
        \end{equation}
        to get the last inequality.
         Thus,
\begin{multline}
							\label{eq1 Lp mathcal U}
\|Du\|_{L_p(\bR^{d+1})}^p
             +\lambda^{p/2}\|u\|_{L_p(\bR^{d+1})}^p \leq N_0\|u\|_{L_p(\bR^{d+1})}
\\
+ N\int_{\bR^d}\left(\|D(u\zeta(\cdot-y))\|_{L_p(\bR^{d+1})}^p +N\lambda^{p/2}\|u\zeta(\cdot-y)\|_{L_p(\bR^{d+1})}^p\right)\,dy.
\end{multline}
Notice that $u(t,x)\zeta(x-y)$ satisfies
         \[
         \cP_\lambda(u\zeta(\cdot-y))
         =
         -D_i(a_{ij}uD_j\zeta(\cdot-y))
         -a_{ij}D_juD_i\zeta(\cdot-y)
         \]
         \[
         +D_t^{1/2}(h\zeta(\cdot-y))
         +D_i(g_i\zeta(\cdot-y))-g_iD_i\zeta(\cdot-y)
         +f\zeta(\cdot-y)\]  
    in $\bR^{d+1}$.
By applying Lemma \ref{Lp small support}, we have
\begin{multline}
							\label{eq3 Lp mathcal U}
\|D(u\zeta(\cdot-y))\|_{L_p(\bR^{d+1})}^p +\lambda^{p/2}\|u\zeta(\cdot-y)\|_{L_p(\bR^{d+1})}^p
\\
\le N\|uD\zeta(\cdot-y)\|_{L_p(\bR^{d+1})}^p +N\lambda^{-p/2}\|DuD\zeta(\cdot-y)\|_{L_p(\bR^{d+1})}^p
\\
+N\|h\zeta(\cdot-y)\|_{L_p(\bR^{d+1})}^p +N\|g\zeta(\cdot-y)\|_{L_p(\bR^{d+1})}^p
\\
+N\lambda^{-p/2}\|gD\zeta(\cdot-y)\|_{L_p(\bR^{d+1})}^p +N\lambda^{-p/2}\|f\zeta(\cdot-y)\|_{L_p(\bR^{d+1})}^p.
\end{multline}
By combining \eqref{eq1 Lp mathcal U} and \eqref{eq3 Lp mathcal U} and using \eqref{eq2 Lp mathcal U} and \eqref{eq4 Lp mathcal U}, we arrive at 
    \[
     \|Du\|_{L_p(\bR^{d+1})}^p
             +\lambda^{p/2}\|u\|_{L_p(\bR^{d+1})}^p
    \]
    \[
    \le N_0\|u\|_{L_p(\bR^{d+1})}^p
    +N_0\lambda^{-p/2}\|Du\|_{L_p(\bR^{d+1})}^p
    +N_0\lambda^{-p/2}\|g\|_{L_p(\bR^{d+1})}^p
    \]
    \[
    +N\|h\|_{L_p(\bR^{d+1})}^p
    +N\|g\|_{L_p(\bR^{d+1})}^p
    +N\lambda^{-p/2}\|f\|_{L_p(\bR^{d+1})}^p.
    \]
    By selecting a sufficiently large $\lambda_0$ such that 
    \[
    N_0\lambda_0^{-p/2}\le \frac{1}{2},
    \]
    we see that for $\lambda\ge \lambda_0$, the estimate \eqref{est Lp mathcal U} holds true.
    \end{proof}
    
    In the case where $D_i(a_{ij}D_ju)=\Delta u$, we derive $L_p$-estimates for $D_t^{1/2}u$, $Du$, and $\sqrt{\lambda}u$.
Essentially, for the Laplacian case, one can directly estimate not only $Du$ and $u$, but also $D_t^{1/2}u$. It is worth noting that the assumption on the support of $u$ is not necessary in this case, as we do not employ the freezing coefficients technique.

\begin{theorem}
    \label{Lp laplacian}
    Let $p\in (2,\infty)$, $\lambda \geq 0$, and $h, g_i, f \in L_p(\bR^{d+1}) \cap L_2(\bR^{d+1})$ with $f \equiv 0$ if $\lambda = 0$.
There exist positive constants $\gamma$ and $N$ depending only on $d$, $\delta$, and $p$ such that, under Assumption \ref{assum coeffi} ($\gamma$), for any  $u \in H^{1/2,1}_p(\bR^{d+1}) \cap H^{1/2,1}_2(\bR^{d+1})$ satisfying 
    \[
     u_t-\Delta u+\lambda u=D_t^{1/2}h+D_ig_i+f \quad \text{in} \,\, \bR^{d+1}, 
    \]
we have
\begin{equation}
							\label{eq0715_01}
\|D_t^{1/2}u\|_p + \|Du\|_p + \sqrt{\lambda}\|u\|_p \le N \left( \|h\|_p + \|g_i\|_p +\lambda^{-1/2} \|f\|_p \right),
\end{equation}
where $\|\cdot\|_p = \|\cdot\|_{L_p(\bR^{d+1})}$.

\end{theorem}

\begin{proof}
Using Lemma \ref{mean oscillation U simple} with the argument in the proof of Lemma \ref{small support}, we see that 
\[
\Phi^{\#}_{dy}(X) \le N\kappa^{1+d/2}(\cS\cM (|F|^2)(X))^{1/2}+N\kappa^{-1} (\cM(|U|^2)(X))^{1/2}
\]
for all $X\in \bR^{d+1}$, where $F$ and $U$ are defined in \eqref{eq1010_01} and \eqref{eq1010_02}, respectively.
Here, $\Phi$ can be any one of $D_t^{1/2}u$, $Du$, and $\sqrt{\lambda}u$.
We conclude the proof by using the Fefferman-Stein sharp function theorem and the Hardy-Littlewood maximal function theorem, taking $\kappa$ sufficiently large so that the term $\|U\|_{p}$ can be absorbed into the left-hand side of the estimate \eqref{eq0715_01}.
\end{proof}

\begin{remark}
Since, in Theorem \ref{Lp laplacian}, the operator is the heat operator and the domain is $\bR^{d+1}$, to prove \eqref{eq0715_01}, especially to prove
\[
\|D_t^{1/2}u\|_{L_p(\bR^{d+1})} \lesssim \||h|+|g_i| + \lambda^{-1/2}|f|\|_{L_p(\bR^{d+1})},
\]
one may use the well-known multiplier theorem with respect to $(t,x) \in \bR^{d+1}$.
Here, we use mean oscillation estimates, which are also applicable to domains such as $\bR \times \Omega$, where $\Omega$ is not necessarily the whole Euclidean space.
On the other hand, if we have estimates of $Du$ and $\sqrt{\lambda}u$ (not $D_t^{1/2}u$) for equations as in \eqref{main eq} {\em without} the $D_t^{1/2}h$ term on the right-hand side, which are indeed available in the literature (see, for instance, \cite{MR2764911, MR2771670}) under certain conditions on domains and coefficients, Theorem \ref{Lp laplacian}, combined with the estimates for $Du$ and $\sqrt{\lambda}u$, can yield the main results in this paper.
We do not pursue this direction because we aim to demonstrate how to obtain the desired estimates for equations as in \eqref{main eq} {\em with} the $D_t^{1/2}h$ term on the right-hand side, without relying on known results for the usual parabolic equations (i.e. equations as in \eqref{main eq} without the $D_t^{1/2}h$ term).
Depending on the conditions on domains and coefficients, estimates of $Du$ and $\sqrt{\lambda}u$ for equations as in \eqref{main eq} without the $D_t^{1/2}h$ term are not available.
In some cases, our results can be used to derive estimates of $Du$ and $\sqrt{\lambda}u$ for the usual parabolic equations. 
\end{remark}

We are ready to complete the proof of Theorem \ref{main whole space}.  

\begin{proof}[Proof of Theorem \ref{main whole space}]
First, we consider the case $p\in (2,\infty)$.
Temporally, we assume that 
\begin{equation}
							\label{eq0905_01}
u \in H_p^{1/2,1}(\bR^{d+1}) \cap H_2^{1/2,1}(\bR^{d+1}), \quad h, g_i, f \in L_p(\bR^{d+1}) \cap L_2(\bR^{d+1}).
\end{equation}
Note that 
\begin{equation}
							\label{manipulated eqn}
u_t-\Delta u+\lambda u= D_t^{1/2}h+D_i\tilde{g}_i+f \quad \text{in} \,\, \bR^{d+1},
\end{equation}
where $\tilde{g}_i=g_i+(a_{ij}-\delta_{ij})D_ju$ and $\delta_{ij}$ is the Kronecker delta.
By using Theorem \ref{Lp laplacian}, we have 
\[
\|U\|_{L_p(\bR^{d+1})}\le N\|F\|_{L_p(\bR^{d+1})}+N\|Du\|_{L_p(\bR^{d+1})},
\]
where $U$ is as in \eqref{eq1010_02} and $F$ is as in \eqref{eq1010_01}.
Furthermore, from Theorem \ref{Lp mathcal U}, we also have 
\begin{equation*}
\|Du\|_{L_p(\bR^{d+1})}\le N\|F\|_{L_p(\bR^{d+1})}.
\end{equation*}
Combining the above two assumptions, we obtain the estimate \eqref{main whole space estimate} under the additional assumptions in \eqref{eq0905_01}.

To prove \eqref{main whole space estimate} without those in \eqref{eq0905_01}, let $\zeta \in C_0^\infty(\bR^d)$ be such that
\[
\zeta(x) = \left\{
\begin{aligned}
1 \quad &\text{for} \,\, |x| \leq 1,
\\
0 \quad &\text{for} \,\, |x| \geq 2.
\end{aligned}
\right.
\]
Set $\zeta_k(x) = \zeta(x/2^k)$.
Also, take $\eta_k(t)$ from \eqref{eq0513_02}.
We then define
\[
w^k(t,x):=\eta_k(t)\zeta_k(x) u(t,x),
\]
which satisfies $w^k \in H_p^{1/2,1}(\bR^{d+1}) \cap H_2^{1/2,1}(\bR^{d+1})$ and
\begin{equation}
							\label{eq0905_03}
w^k \to u \quad \text{in} \,\, H_p^{1/2,1}(\bR^{d+1}).
\end{equation}
Indeed, it is clear that $w^k, Dw^k \in L_p(\bR^{d+1}) \cap L_2(\bR^{d+1})$.
To check $D_t^{1/2}w^k \in L_p(\bR^{d+1})$ and \eqref{eq0905_03}, see Lemmas \ref{lem0513_1} and \ref{tail estimate}.
To verify $D_t^{1/2}w^k \in L_2(\bR^{d+1})$, observe that
\[
\sum_{j=1}^\infty 2^{-j}\|\zeta_k u\|_{L_2\left((-2^{k+j},2^{k+j}) \times \bR^d\right)} \leq \sum_{j=1}^\infty 2^{-j}\|u\|_{L_2\left((-2^{k+j},2^{k+j}) \times B_{2^{k+1}}\right)}
\]
\[
\leq N(d,p)2^{k(d+1)(1/2-1/p)} \sum_{j=1}^\infty 2^{-j(1/2+1/p)} \|u\|_{L_p(\bR^{d+1})} < \infty.
\]
Hence, if we set
\[
(\zeta_k u)_k := \frac{1}{\sqrt{8\pi}} \int_\bR \zeta_k(x) u(t+\ell,x) \frac{\eta_k(t+\ell) - \eta_k(t)}{|\ell|^{3/2}} \, d\ell,
\]
by Lemma \ref{tail estimate}, we have $(\zeta_k u)_k \in L_2(\bR^{d+1})$.
From this with Lemma \ref{lem0513_1} it follows that
\[
D_t^{1/2} w^k = \eta_k(t) \zeta_k(x) D_t^{1/2}u +  (\zeta_k u)_k \in L_2(\bR^{d+1}).
\]

Denote
\[
h^k = \eta_k \zeta_k h, \quad
g_i^k = \eta_k \zeta_k g_i - a_{ij} u \eta_k D_j \zeta_k,
\]
\[
f^k = \eta_k\zeta_k f - g_i \eta_k D_i \zeta_k - (\zeta_k h)_k + \eta_k' \zeta_k u - a_{ij} \eta_k D_i\zeta_k D_j u,
\]
where
\[
(\zeta_k h)_k = \frac{1}{8\pi} \int_\bR \zeta_k(x) h(t+\ell,x)\frac{\eta_k(t+\ell)-\eta_k(t)}{|\ell|^{3/2}} \, d\ell.
\]
Observe that
\[
h^k, g_i^k, f^k \in L_p(\bR^{d+1}) \cap L_2(\bR^{d+1}),
\]
and
\begin{equation}
							\label{eq0905_02}
h^k \to h, \quad g_i^k \to g_i, \quad f^k \to f \quad \text{in} \,\, L_p(\bR^{d+1})
\end{equation}
as $k \to \infty$.
In particular, by Lemma \ref{tail estimate} and following the argument above used to verify that $(\zeta_ku)_k \in L_2(\bR^{d+1})$, we have $\|(\zeta_k h)_k\|_{L_p(\bR^{d+1})} \to 0$ and $(\zeta_k h)_k \in L_2(\bR^{d+1})$.

Since $u$ satisfies \eqref{main eq}, we see that $w^k$ satisfies 
\[
w^k_t - D_i(a_{ij} D_j w^k) + \lambda w^k = D_t^{1/2} h^k + D_i g_i^k + f^k \quad \text{in} \,\, \bR^{d+1}.
\]
Given that $w^k, h^k, g_i^k, f^k$ satisfy the assumptions as in \eqref{eq0905_01}, the proof above provides the estimate \eqref{main whole space estimate} for $w^k$.
By recalling \eqref{eq0905_03} and \eqref{eq0905_02}, and letting $k \to \infty$, we finally obtain the estimate \eqref{main whole space estimate} for $u$ when $p \in (2,\infty)$.

To prove the solvability, since we have already proved the estimate \eqref{main whole space estimate}, it is enough to solve
\begin{equation}
							\label{eq0523_03}
u_t - \Delta u + \lambda u = D_t^{1/2}h + D_i g_i + f
\end{equation}
in $\bR^{d+1}$, where $h, g_i, f \in C_0^\infty(\bR^{d+1})$.
Then, one can say that $D_t^{1/2}h + D_i g_i + f \in L_p(\bR^{d+1})$ (see Remark \ref{rem0515_1}), so that by the well-known $L_p$ theory, there exists $u \in L_p(\bR^{d+1})$ such that $u_t, Du, D^2u \in L_p(\bR^{d+1})$ satisfying \eqref{eq0523_03} a.e. in $\bR^{d+1}$.
By Lemma \ref{lem0515_1}, $D_t^{1/2}u \in L_p(\bR^{d+1})$; thus $u \in H_p^{1/2,1}(\bR^{d+1})$, and using Lemma \ref{lem0514_1} (4), $u$ is also a solution to \eqref{eq0523_03} in the sense of \eqref{eq0523_04}.

For the case $p\in (1,2)$, we use the duality argument. Let $q$ be the conjugate exponent of $p$.
Let $\tilde{h}\in L_q(\bR^{d+1})$, $\tilde{g}_i \in L_q(\bR^{d+1})$, and $\tilde{f}\in L_q(\bR^{d+1})$.
By using Theorem \ref{main whole space} for $q\in(2,\infty)$ proved above with the change of variable $t\to -t$, we find a unique solution $\tilde{u}$ of the equation
\[
-\tilde{u}_t-D_j(a_{ij}D_i \tilde{u})+\lambda \tilde{u}=D_t^{1/2}\tilde{h}-D_i\tilde{g_i}+\tilde{f}\quad \text{in} \,\, \bR^{d+1}
\]
with the estimate 
\[
\|\tilde{U}\|_{L_q(\bR^{d+1})}
\le N\|\tilde{F}\|_{L_q(\bR^{d+1})},
\]
where $\tilde{U}$ and $\tilde{F}$ are defined as $U$ and $F$, using $\tilde{u}$, $\tilde{h}$, $\tilde{g}_i$, and $\tilde{f}$.
Note that by Fourier transforms with the help of Remark \ref{rem0506_1},
\[
\int_{\bR^{d+1}} H(D_t^{1/2}\tilde{u}) \, D_t^{1/2} u \, dX = - \int_{\bR^{d+1}} H(D_t^{1/2}u) \, D_t^{1/2} \tilde{u} \, dX.
\]
Then, we see that
\[
\int_{\bR^{d+1}}\tilde{h}D_t^{1/2}u+ \tilde{g_i}D_iu+\tilde{f}u\,dX
\]
\[
=\int_{\bR^{d+1}} {H(D_t^{1/2}\tilde{u}) \, D_t^{1/2}u} \, +a_{ij}D_i\tilde{u}D_ju+\lambda u\tilde{u}\,dX
\]
\[
=\int_{\bR^{d+1}}hD_t^{1/2}\tilde{u}-g_iD_i\tilde{u}+f\tilde{u}\,dX
\le\|\tilde{U}\|_{L_q(\bR^{d+1})}\|F\|_{L_p(\bR^{d+1})}
\]
\[
\le N\|\tilde{F}\|_{L_q(\bR^{d+1})}\|F\|_{L_p(\bR^{d+1})}.
\]
By the duality, we reach that 
\begin{equation*}
\|U\|_{L_p(\bR^{d+1})}\le \|F\|_{L_p(\bR^{d+1})}.
\end{equation*}

The solvability is proved as in the case for $p \in (2,\infty)$.
Finally, Theorem \ref{Interior L_2 solvability} covers the case $p=2$.
This ends the proof.
\end{proof}

\section{Mean oscillation estimates: $x_1$-direction measurable $a_{ij}$, and Proof of Theorem \ref{thm0521_1}}\label{sec x_1}

In this section, we prove Theorem \ref{thm0521_1} under Assumption \ref{assum0521_1}.
We begin with mean oscillation estimates of solutions to homogeneous equations with the operator having coefficients $a_{ij}=a_{ij}(x_1)$, which, in contrast to Section \ref{sec05},  are functions of only $x_1 \in \bR$ without any regularity conditions.
Then, we obtain mean oscillation estimates of solutions to equations when the coefficients $a_{ij}$ are functions of only $x_1 \in \bR$, not on the whole cylinder, but only on $Q_{\kappa r}$ (see Lemma \ref{mean oscillation x_1}).
From this, we obtain mean oscillation estimates for solutions to equations with coefficients satisfying Assumption \ref{assum0521_1} (see Lemma \ref{mean oscillation general coeffi x_1}).
Finally, we conclude the section with a brief proof of Theorem \ref{thm0521_1}.
Throughout the section, set
\[
\cU'=( D'u,\sqrt{\lambda}u) \quad \text{and} \quad \Theta_u:= \sum_{j=1}^d a_{1j}D_j u,
\]
where $D'=(D_2,\dots, D_d)$.

\begin{lemma}
    \label{holder x_1 mble}
   Let $r\in(0,\infty)$, $\kappa \in [2,\infty)$, and $\lambda\ge0$. Assume that $u\in C^\infty_{loc}(\bR^{d+1})$ satisfies 
    \[
    u_t-D_i(a_{ij}(x_1)D_j u)+\lambda u=0 \quad \text{in} \,\, Q_{\kappa r}.
    \]
    Then, we have 
    \[
    (|\cU'-(\cU')_{Q_r}|)_{Q_r}+(|\Theta_u-(\Theta_u)_{Q_r}|)_{Q_r}
    \le 
    N\kappa^{-1/2}(|Du|^2+\lambda|u|^2)_{Q_{\kappa r}}^{1/2}.
    \]
    where $N=N(d,\delta)$.
\end{lemma}

\begin{proof}
    We refer to \cite[Lemma 5.1]{MR2835999}.
\end{proof}

Note that in the lemma below, the coefficients $a_{ij}$ are assumed to be $a_{ij}(x_1)$, which are functions of $x_1 \in \bR$, only on $Q_{\kappa r}(X_0)$.

\begin{lemma}\label{mean oscillation x_1}
        Let $r\in(0,\infty)$, $\kappa \in [4,\infty)$, $\lambda\ge0$, and $X_0=(t_0,x_0)\in \bR^{d+1}$. Suppose that $u\in H^{1/2,1}_2(\bR^{d+1})$ satisfies 
    \[
  u_t-D_i(a_{ij}D_ju)+\lambda u=D_t^{1/2}h+D_ig_i+f \quad \text{in} \,\, \bR\times B_{\kappa r}(x_0), 
    \]
    where $h, g_i, f \in L_2(\bR^{d+1})$ with $f \equiv 0$ if $\lambda = 0$, and 
\[
a_{ij}(t,x) = a_{ij}(x_1) \quad \text{in} \,\, Q_{\kappa r}(X_0).
\]
Then, we have 
        \[
 (|\cU'-(\cU')_{Q_r(X_0)}|)_{Q_r(X_0)}+(|\Theta_u-(\Theta_u)_{Q_r(X_0)}|)_{Q_r(X_0)}
 \]\[
    \le N\kappa^{1+d/2}\sum_{j=0}^\infty2^{-j/4}(|F|^2)^{1/2}_{Q_{2^{j/2}\kappa r,\kappa r}(X_0)}+N\kappa^{-1/2} (|\cU|^2)^{1/2}_{Q_{\kappa r}(X_0)},
    \]
    where $N=N(d,\delta)$, where $F$ and $\cU$ are as in \eqref{eq1010_01}.
\end{lemma}

\begin{proof}
We may assume that $a_{ij}(x_1)$ are smooth.
Thanks to Theorem \ref{Interior L_2 solvability}, we can also assume that $h, g_i, f \in C_0^\infty(\bR^{d+1})$.
By utilizing dilation and translation, we also assume that $r=4/\kappa$ and $X_0=0$.  By taking $\lambda \searrow 0$, we only consider the case $\lambda>0$.  
    By Theorem \ref{Interior L_2 solvability}, we take a unique solution $w\in \mathring H_2^{1/2,1}(\bR\times B_{\kappa r})$ satisfying 
    \[
      w_t-D_i(a_{ij}D_j u)+\lambda u=D_t^{1/2} h+D_ig_i+f \quad \text{in} \,\, \bR\times B_{\kappa r}.
    \]
       Set $v=u-w$, which belongs to $H_2^{1/2,1}(\bR \times B_{\kappa r})$ and satisfies
\[
v_t-D_i(a_{ij}D_j v)+\lambda v = 0 \quad \text{in} \,\, \bR \times B_{\kappa r}.
\]
Denote
    \[
    \cW=(Dw,\sqrt{\lambda}w),
    \quad \cW'=(D'w,\sqrt{\lambda}w),
     \quad  \Theta_w=\sum_{j=1}^d a_{1j}D_jw,\]
    and
      \[
     \cV=(Dv,\sqrt{\lambda}v),
    \quad \cV'=(D'v,\sqrt{\lambda}v), 
    \quad  \Theta_v=\sum_{j=1}^d a_{1j}D_jv.
    \]
 
    As in the proof of Lemma \ref{mean oscillation cal U}, using H\"older's inequality and Lemma \ref{local L_2 estimate lemma}, we have 
    \[
    (|\cW'-(\cW')_{Q_r}|)_{Q_r}+ (|\Theta_w-(\Theta_w)_{Q_r}|)_{Q_r}
       \]
           \begin{equation}
        \label{estimate w whole x_1}
    \le    N (|\cW|^2)_{Q_r}^{1/2}
    \le N\kappa^{1+d/2}\sum_{j=0}^\infty 2^{-j/4}(|F|^2)^{1/2}_{Q_{2^{j/2}\kappa r,\kappa r}}.
         \end{equation}
    By using Lemma \ref{holder x_1 mble} with $r=4/\kappa$ (since $a_{ij}$ are infinitely differentiable and $h, g_i, f \in C_0^\infty(\bR^{d+1})$, by the well-known theory $v \in C_{loc}^\infty(\bR^{d+1)}$), we see that
    \[
    (|\cV'-(\cV')_{Q_r}|)_{Q_r}+ (|\Theta_v-(\Theta_v)_{Q_r}|)_{Q_r}
    \le 
    N\kappa^{-1/2}(|\cV|^2)_{Q_{\kappa r}}^{1/2}
    \]
    \begin{equation}
        \label{estiamte v whole x_1}
        \le 
    N\kappa^{-1/2}\sum_{j=0}^\infty2^{-j/4} (|F|^2)^{1/2}_{Q_{2^{j/2}\kappa r,\kappa r}}+N\kappa^{-1/2} (|\cU|^2)^{1/2}_{Q_{\kappa r}},
        \end{equation}
where we use $v=u-w$ and Lemma \ref{local L_2 estimate lemma} to get the last inequality.
Combining the estimates \eqref{estimate w whole x_1} and \eqref{estiamte v whole x_1} gives
\[
(|\cU'-(\cU')_{Q_r}|)_{Q_r}+(|\Theta_u-(\Theta_u)_{Q_r}|)_{Q_r}
\]
\[
\le N\kappa^{1+d/2}\sum_{j=0}^\infty 2^{-j/4}(|F|^2)^{1/2}_{Q_{2^{j/2}\kappa r,\kappa r}}
  +N\kappa^{-1/2} (|\cU|^2)^{1/2}_{Q_{\kappa r}}.
    \]
The lemma is proved.
\end{proof}

Now, we obtain mean oscillation estimates for the operator,
\[\cP_\lambda u=u_t-D_i(a_{ij}D_j u)+\lambda u,\]
where the coefficients satisfy Assumption \ref{assum0521_1}.
As in the proof of Lemma \ref{mean oscillation general coeffi}, we use the method of freezing coefficients.
However, unlike in the proof of Lemma \ref{mean oscillation general coeffi}, the freezing coefficients are $a_{ij}(t,x_1)$, which are measurable in $x_1 \in \bR$ and piecewise constant in $t \in \bR$. 
Here, the constant $R_0$ is from Assumption \ref{assum0521_1}.
 
\begin{lemma}\label{mean oscillation general coeffi x_1}
Let $\nu\in (2,\infty)$, $\nu'=2\nu/(\nu-2)$, $r\in(0,R_0/\kappa]$, $\kappa \in [4,\infty)$, $\lambda\ge0$, and $X_0=(t_0,x_0)\in \bR^{d+1}$.
Assume that $u \in H_{\nu}^{1/2,1}(\bR^{d+1}) \cap H_2^{1/2,1}(\bR^{d+1})$ satisfies 
    \[
    \cP_\lambda u=D_t^{1/2}h+D_ig_i+f \quad \text{in} \,\, \bR^{d+1}, 
    \]
    where $h, g_i, f \in L_\nu(\bR^{d+1}) \cap L_2(\bR^{d+1})$ with $f \equiv 0$ if $\lambda = 0$.
    Then,
    under Assumption \ref{assum0521_1} ($\gamma$), there exists a function $\cU^Q$ on $Q:=Q_r(X_0)$ such that
\begin{equation}
							\label{eq1010_03}
N^{-1}(|Du| + \sqrt{\lambda}|u|) \le |\cU^Q| \le N(|Du| + \sqrt{\lambda}|u|) \quad \text{on}\,\, Q
\end{equation}
and
\begin{multline*}
(|\cU^Q-(\cU^Q)_{Q_r(X_0)}|)_{Q_r(X_0)}\le N\kappa^{1+d/2}\sum_{j=0}^\infty2^{-j/4}(|F|^2)^{1/2}_{Q_{2^{j/2}\kappa r,\kappa r}(X_0)}
\\
+N\kappa^{-1/2} (|\cU|^2)^{1/2}_{Q_{\kappa r}(X_0)}+N\gamma^{1/\nu'}\kappa^{1+d/2}\sum_{j=0}^\infty2^{-j/4}(|Du|^\nu)_{Q_{2^{j/2}R,R}(X_0)}^{1/\nu},
\end{multline*}
where $N=N(d,\delta,\nu)$, where $F$ and $\cU$ are as in \eqref{eq1010_01}. 
\end{lemma}

\begin{proof}
Let $R=\kappa r$. We may assume that $X_0=0$.
Let $\tau_k = 2kR^2$, $k=0,1,2,\ldots$, and
\[
\bar{a}_{ij}^k(x_1) = \dashint_{Q_R'(\tau_k,0)} a_{ij}(s,x_1,y') \, dy' \, ds.
\]
Then, we set
\[
\bar{a}_{ij}(t,x_1) = \bar{a}^k_{ij}(x_1) \quad \text{for} \,\, t \in (\tau_k-R^2, \tau_k+R^2]
\]
and
\[
\overline{\cP_\lambda}u=u_t-D_i(\bar{a}_{ij}(t, x_1)D_j u)+\lambda u.
\]
We have 
\begin{equation}
        \label{bar eq x_1}
    \overline{\cP_\lambda}u=D_t^{1/2}h+D_i\tilde{g}_i+h \quad \text{in} \,\, \bR^{d+1}, 
        \end{equation}
    where 
    \[
    \tilde{g}_i=g_i+(a_{ij}-\bar{a}_{ij})D_j u.
    \]
Notice that
\[
\bar{a}_{ij}(t,x_1)= {\bar{a}_{ij}^0(x_1)} \quad \text{in} \,\, Q_{R}=Q_{\kappa r}.
\]
By applying Lemma \ref{mean oscillation x_1} to $u$ with the equation \eqref{bar eq x_1}, we have 
    \[
    (|\cU'-(\cU')_{Q_r}|)_{Q_r}+ (|\Theta_u-(\Theta_u)_{Q_r}|)_{Q_r}
    \le N\kappa^{1+d/2}\sum_{l=0}^\infty2^{-l/4}(|F|^2)^{1/2}_{Q_{2^{l/2}R,R}}
    \]
    \[
+N\kappa^{-1/2} (|\cU|^2)^{1/2}_{Q_R}
    +N\kappa^{1+d/2}\sum_{l=0}^\infty2^{-l/4}(|a_{ij}-\bar{a}_{ij}|^2|Du|^2)^{1/2}_{Q_{2^{l/2}R,R}}.
    \]
By H\"older's inequality, $|a_{ij}|\le \delta^{-1}$, Assumption \ref{assum0521_1} ($\gamma$), and the choice of $\bar{a}_{ij}$,  it follows that
\[
(|a_{ij}-\bar{a}_{ij}|^2|Du|^2)_{Q_{2^{l/2}R,R}}^{1/2}
\le (|a_{ij}-\bar{a}_{ij}|^{\nu'})_{Q_{2^{l/2}R,R}}^{1/\nu'}
(|Du|^\nu)_{Q_{2^{l/2}R,R}}^{1/\nu}
 \]
\[
\le  N(|a_{ij}-\bar{a}_{ij}|)_{Q_{2^{l/2}R,R}}^{1/\nu'}
(|Du|^\nu)_{Q_{2^{l/2}R,R}}^{1/\nu} \leq N\gamma^{1/\nu'}(|Du|^\nu)_{Q_{2^{l/2}R,R}}^{1/\nu}
\]
for any $\ell = 0, 1,\ldots$, where for the last inequality, see Remark \ref{rem0522_02}.

From this, we obtain that 
\[
     (|\cU'-(\cU')_{Q_r}|)_{Q_r}+ (|\Theta_u-(\Theta_u)_{Q_r}|)_{Q_r}
    \le N\kappa^{1+d/2}\sum_{j=0}^\infty2^{-j/4}(|F|^2)^{1/2}_{Q_{2^{j/2}\kappa r,\kappa r}}
    \]
    \[+N\kappa^{-1/2} (|\cU|^2)^{1/2}_{Q_{\kappa r}}
    +N\gamma^{1/\nu'}\kappa^{1+d/2}\sum_{j=0}^\infty2^{-j/4}(|Du|^\nu)_{Q_{2^{j/2}R,R}}^{1/\nu}.
    \]
    We set $\cU^Q=(\cU',\Theta_u)$. Then by the ellipticity condition, we have \eqref{eq1010_03}.
This ends the proof.
\end{proof}

We now prove Theorem \ref{thm0521_1}.
\begin{proof}[Proof of Theorem \ref{thm0521_1}]
Essentially, using Lemma \ref{mean oscillation general coeffi x_1} instead of \ref{mean oscillation general coeffi}, the proof remains unchanged from that of Theorem \ref{main whole space}.
It is worth noting that we use a version of the Fefferman-Stein sharp function theorem, \cite[Corollary 2.8]{MR3812104}, to obtain the $L_p$-estimate of $Du$ and $\sqrt{\lambda}u$.
\end{proof}

\bibliographystyle{plain}

\def\cprime{$'$}

\end{document}